\pdfoutput=1

\documentclass{amsart}

\usepackage[mathscr]{eucal}
\usepackage{amssymb}
\usepackage[usenames,dvipsnames]{xcolor} 
\usepackage[normalem]{ulem}
\usepackage{amsthm}
\usepackage{bbold}
\usepackage{comment}
\usepackage{enumitem}
\usepackage{amsmath}
\usepackage{wasysym}
\usepackage{tikz-cd}
\usepackage{calc}

\usepackage{etoolbox} 

\usepackage[unicode]{hyperref} 

\definecolor{dark-red}{rgb}{0.5,0.15,0.15}
\definecolor{dark-blue}{rgb}{0.15,0.15,0.6}
\definecolor{dark-green}{rgb}{0.15,0.6,0.15}
\hypersetup{
    colorlinks, linkcolor=OliveGreen,
    citecolor=OliveGreen, urlcolor=OliveGreen
}

\usepackage[nameinlink,capitalise,noabbrev]{cleveref}
\usepackage[hyphenbreaks]{breakurl}

\numberwithin{equation}{section}
\usepackage[all]{xy}
\xyoption{line}
\usepackage{graphicx}
\usepackage{mathtools}
\SelectTips{cm}{10}

\newtheorem{Thm}[equation]{Theorem}
\newtheorem*{Thm*}{Theorem}
\newtheorem{Prop}[equation]{Proposition}
\newtheorem{Lem}[equation]{Lemma}
\newtheorem{Cor}[equation]{Corollary}

\newtheorem{thmx}{Theorem}

\theoremstyle{remark}
\newtheorem{Def}[equation]{Definition}

\newtheorem{Rec}[equation]{Recollection}
\newtheorem{Not}[equation]{Notation}
\newtheorem{Exa}[equation]{Example}

\newtheorem{Conv}[equation]{Convention}

\newtheorem{Warn}[equation]{Warning}

\newtheorem{Rem}[equation]{Remark}

\usepackage{tikz}
\usepackage{tikz-cd}
\usetikzlibrary{trees}
\usetikzlibrary[shapes]
\usetikzlibrary[arrows]
\usetikzlibrary{patterns}
\usetikzlibrary{fadings}
\usetikzlibrary{backgrounds}
\usetikzlibrary{decorations.pathreplacing}
\usetikzlibrary{decorations.pathmorphing}
\usetikzlibrary{positioning}
\usetikzlibrary{shapes.geometric}

\tikzstyle{category} = [rectangle, rounded corners, minimum width=3cm, minimum height=1cm, text centered, text width=2.5cm, draw=black]

\tikzstyle{arrow} = [thick,->,>=stealth]
\tikzstyle{arrow2} = [thick,->,>=stealth,dotted]

\newcommand{\nc}{\newcommand}
\nc{\dmo}{\DeclareMathOperator}

\usepackage{todonotes}

\nc{\overbar}[1]{\mkern 1.5mu\overline{\mkern-1.5mu#1\mkern-1.5mu}\mkern 1.5mu}

\nc{\kappaaux}{g}
\nc{\kappam}{{\kappaaux({\frak m})}}
\nc{\kappaP}{{\kappaaux(\cat P)}}
\nc{\kappaQ}{{\kappaaux(\cat Q)}}
\nc{\kappaSP}{{\kappaaux_{\cat S}(\cat P)}}
\nc{\kappaTP}{{\kappaaux_{\cat T}(\cat P)}}
\nc{\kappaSQ}{{\kappaaux_{\cat S}(\cat Q)}}
\nc{\kappaTQ}{{\kappaaux_{\cat T}(\cat Q)}}
\nc{\kappaphiB}{{\kappaaux(\varphi(\cat B))}}
\nc{\kappaphiQ}{{\kappaaux(\varphi(\cat Q))}}

\dmo{\Sub}{Sub}
\nc{\SpEn}{\cat S_{E(n)}}
\nc{\SpEnf}{\cat S_n}
\nc{\Loco}[1]{\Loc_{\otimes}\hspace{-0.3ex}\langle #1 \rangle}
\nc{\bbullet}{{\scriptscriptstyle\hspace{-1pt}\bullet}}
\nc{\bullett}{{\scriptscriptstyle\bullet}\hspace{-1pt}}
\nc{\LF}{L\hspace{-0.2ex}F}
\nc{\SpG}{\Sp_G}
\nc{\SpGn}{\Sp_{G,n}}
\nc{\EG}{\bbE_G}
\nc{\EH}{\bbE_H}
\nc{\DEG}{\Der(\EG)}
\nc{\DEH}{\Der(\EH)}
\nc{\DE}{\Der(\bbE)}
\nc{\Prst}{{\cat P}\mathrm{r^{st}}}
\nc{\Mack}[2]{\mathrm{Mack}_{#1}(#2)}
\nc{\SC}{S\cat C}
\dmo{\fin}{{fin}}
\dmo{\DM}{DM}
\dmo{\fp}{fp}
\nc{\DMQ}{\DM_Q}
\dmo{\DerKal}{DMack}
\dmo{\Der}{D}
\dmo{\DMot}{DMot}
\dmo{\rmH}{H}
\dmo{\piu}{\underline{\pi}}
\dmo{\Sphere}{\mathbb{S}}
\nc{\HA}{{\rmH \hspace{-0.2em}\bbA}}
\nc{\HZ}{{\rmH \hspace{-0.2em}\bbZ}}
\nc{\HZbar}{{\rmH \hspace{-0.2em}\underline{\bbZ}}}
\nc{\Fp}{{\bbF_{\hspace{-0.1em}p}}}
\nc{\HFp}{{\rmH \hspace{-0.15em}\bbF_{\hspace{-0.1em}p}}}
\nc{\DHZG}{\Der(\HZ_G)}
\nc{\DHZH}{\Der(\HZ_H)}
\nc{\DHZK}{\Der(\HZ_K)}
\nc{\DHZGN}{\Der(\HZ_{G/N})}
\nc{\DHZGG}{\Der(\HZ_{G/G})}
\nc{\DHZCp}{\Der(\HZ_{C_p})}
\nc{\DHZGprime}{\Der(\HZ_{G'})}
\nc{\DHZ}{\Der(\HZ)}
\nc{\frakp}{\mathfrak{p}}
\nc{\frakq}{\mathfrak{q}}
\nc{\Z}{\mathbb{Z}}
\nc{\SSG}{\text{sSet}_*^G}
\nc{\sSet}{\text{sSet}}

\dmo{\csupp}{supp_{coh}}
\dmo{\Con}{Conj}
\dmo{\Id}{Id}
\dmo{\Loc}{Loc}
\dmo{\rmK}{\textrm{\rm K}}
\dmo{\Spc}{Spc}
\dmo{\thick}{thick}
\dmo{\Thick}{Thick}
\nc{\Thickt}[1]{\Thick_\otimes\langle #1 \rangle}
\dmo{\cone}{cone}
\dmo{\End}{End}
\dmo{\Mor}{Mor}
\dmo{\Hom}{Hom}
\dmo{\id}{id}
\dmo{\incl}{incl}
\dmo{\Img}{Im}
\dmo{\im}{im}
\dmo{\Ker}{Ker}
\dmo{\ind}{ind}
\dmo{\CoInd}{coind}
\dmo{\res}{res}
\dmo{\infl}{infl}
\dmo{\triv}{triv}
\dmo{\Tel}{Tel} 
\dmo{\grMod}{grMod}%
\dmo{\Mod}{Mod}%
\dmo{\opname}{op}
\dmo{\SH}{SH}
\dmo{\smallb}{b}
\dmo{\Spec}{Spec}
\dmo{\supp}{supp}
\dmo{\Supp}{Supp}
\dmo{\cosupp}{cosupp}
\dmo{\Cosupp}{Cosupp}
\nc{\SHc}{{\SH^c}}
\nc{\SHp}{{\SH_{(p)}}}
\nc{\SHcp}{{\SH^c_{(p)}}}
\nc{\SHG}{\SH(G)}
\nc{\SHGp}{\SH(G)_{(p)}}
\nc{\SHGc}{\SHG^c}
\nc{\SHGcp}{\SHG^c_{(p)}}
\nc{\quadtext}[1]{\quad\textrm{#1}\quad}
\nc{\qquadtext}[1]{\qquad\textrm{#1}\qquad}
\nc{\adj}{\dashv}
\nc{\adjto}{\rightleftarrows}
\nc{\bbL}{\mathbb{L}}
\nc{\bbA}{\mathbb{A}}
\nc{\bbE}{\mathbb{E}}
\nc{\bbN}{\mathbb{N}}
\nc{\bbQ}{\mathbb{Q}}
\nc{\bbZ}{\mathbb{Z}}
\nc{\bbF}{\mathbb{F}}
\nc{\cat}[1]{\mathscr{#1}}
\nc{\ie}{{\sl i.e.}, }
\nc{\into}{\mathop{\rightarrowtail}}
\nc{\inv}{^{-1}}
\nc{\isoto}{\mathop{\overset{\sim}\to}}
\nc{\isotoo}{\mathop{\overset{\sim}\too}}
\nc{\onto}{\mathop{\twoheadrightarrow}}
\nc{\too}{\mathop{\longrightarrow}\limits}
\nc{\mapstoo}{\longmapsto}
\nc{\adh}[1]{\overline{#1}}
\nc{\adhpt}[1]{\adh{\{#1\}}}
\nc{\aka}{{a.\,k.\,a.}\ }
\nc{\calF}{\mathcal{F}}
\nc{\eg}{{\sl e.\,g.}}
\nc{\Homcat}[1]{\Hom_{\cat #1}}
\nc{\hook}{\hookrightarrow}
\nc{\ideal}[1]{\langle #1\rangle}
\nc{\ihom}{{\underline{\hom}}}
\nc{\iHom}{\mathcal{H}\mathrm{om}}
\nc{\Mid}{\,\big|\,}
\nc{\MMod}{\,\text{-}\Mod}%
\nc{\op}{^{\opname}}
\nc{\oto}[1]{\overset{#1}\to}
\nc{\otoo}[1]{\overset{#1}{\,\too\,}}
\nc{\sminus}{\!\smallsetminus\!}
\nc{\poplus}[1]{^{\oplus #1}}%
\nc{\potimes}[1]{^{\otimes #1}}
\nc{\sbull}{{\scriptscriptstyle\bullet}}
\nc{\SET}[2]{\big\{\,#1\Mid#2\,\big\}}
\nc{\SpcK}{\Spc(\cat K)}
\nc{\then}{\Rightarrow}
\nc{\unit}{\mathbb{1}}
\nc{\xra}{\xrightarrow}
\nc{\phigeom}[1]{\widetilde{\Phi}^{#1}}
\nc{\phigeomb}[1]{\Phi^{#1}}
\dmo{\Oname}{O}
\dmo{\proper}{proper}
\dmo{\lenormal}{\unlhd}
\dmo{\lnormal}{\lhd}
\nc{\normal}{\trianglelefteq}
\nc{\Op}{\Oname^p}
\nc{\Oq}{\Oname^q}
\dmo{\Sp}{Sp}
\dmo{\Ho}{Ho}
\dmo{\Fin}{Fin}
\dmo{\add}{add}
\dmo{\Fun}{Fun}
\dmo{\Ext}{Ext}
\dmo{\CAlg}{CAlg}
\dmo{\CMon}{CMon}
\dmo{\CC}{\cat C}
\dmo{\DD}{\cat D}
\dmo{\OO}{\mathcal{O}}
\dmo{\Map}{Map}
\dmo{\Span}{Span}
\dmo{\N}{N}
\dmo{\Cat}{Cat}
\dmo{\colim}{colim}
\dmo{\hocolim}{hocolim}
\dmo{\Ch}{Ch}
\dmo{\A}{\mathbb{A}^{eff}}
\nc{\AGeff}{\mathbb{A}_G^{\mathrm{eff}}}
\nc{\BGeff}{\mathcal{B}_G^{\mathrm{eff}}}
\nc{\BG}{{\mathcal{B}_G}}
\nc{\NBGeff}{{\N}{\BGeff}}
\dmo{\Ab}{Ab}
\dmo{\Set}{Set}
\dmo{\ev}{ev}
\dmo{\Spcl}{Spcl}
\nc{\Funadd}{\Fun_{\add}}
\dmo{\proj}{proj}
\dmo{\cof}{cof}

\dmo{\Coideal}{Coideal}
\dmo{\gen}{gen}
\dmo{\StMod}{StMod}

\dmo{\projmod}{Lat}
\dmo{\rep}{rep}
\dmo{\Rep}{Rep}
\dmo{\Perf}{Perf}
\dmo{\stmod}{stmod}
\dmo{\Ind}{Ind}
\nc{\borel}[1]{\underline{#1}}
\dmo{\coind}{coind}
\dmo{\rank}{rank}
\nc{\tH}{\hat{H}}
\dmo{\Nm}{Nm}
\dmo{\Proj}{Proj}
\dmo{\Inj}{Inj}
\dmo{\dual}{dual}
\dmo{\fg}{fg}
\nc{\cdvr}[2]{{#1}_{#2}^{\wedge}}
\nc{\cA}{\mathcal{A}}
\dmo{\orbit}{Or}

\newcommand\noloc{%
  \nobreak
  \mspace{6mu plus 1mu}
  {:}
  \nonscript\mkern-\thinmuskip
  \mathpunct{}
  \mspace{2mu}
}

\nc{\mT}{\kern-0.5em\mod\kern-0.1em\text{-}\cat{T}^c}
\nc{\MT}{\Mod\kern-0.1em\text{-}\cat{T}}
\newcounter{enum-resume-hack}

\begin{document}


\title[Integral stratification]{Stratifying integral representations \\ of finite groups}

\author{Tobias Barthel}

\date{\today}

\makeatletter
\patchcmd{\@setaddresses}{\indent}{\noindent}{}{}
\patchcmd{\@setaddresses}{\indent}{\noindent}{}{}
\patchcmd{\@setaddresses}{\indent}{\noindent}{}{}
\patchcmd{\@setaddresses}{\indent}{\noindent}{}{}
\makeatother

\address{Tobias Barthel, Max Planck Institute for Mathematics, Vivatsgasse 7, 53111 Bonn, Germany}
\email{tbarthel@mpim-bonn.mpg.de}
\urladdr{https://sites.google.com/view/tobiasbarthel/home}

\begin{abstract}
We classify the localizing tensor ideals of the integral stable module category for any finite group $G$. This results in a generic classification of $\bbZ[G]$-lattices of finite and infinite rank and globalizes the modular case established in celebrated work of Benson, Iyengar, and Krause. Further consequences include a verification of the generalized telescope conjecture in this context, a tensor product formula for integral cohomological support, as well as a generalization of Quillen's stratification theorem for group cohomology. Our proof makes use of novel descent techniques for stratification in tensor-triangular geometry that are of independent interest. 
\vspace{-1em}
\end{abstract}

\subjclass[2020]{18F99, 18G65, 18G80, 20C12, 55P91, 55U35}

\maketitle

\tableofcontents
\vspace{-1em}

\section{Introduction}

The subject of this paper is the integral representation theory of finite groups. We consider representations of a finite group $G$ on free $\bbZ$-modules, not necessarily of finite rank. The classification of such representations is, in general, a wild problem, so that a complete algebraic or geometric parametrization of the isomorphism classes is impossible. The goal of this paper is to instead classify these representations generically, that is up to a coarser equivalence relation than isomorphism. 

Let $\cat O$ be the ring of integers in a number field. We write $\cat O[G]$ for the group algebra and denote by $\projmod(G,\cat O)$ the category of $\cat O[G]$-lattices, i.e., $\cat O[G]$-modules whose underlying $\cat O$-module is projective, equipped with the $\cat O$-linear symmetric monoidal structure. The quotient of $\projmod(G,\cat O)$ by the projective $\cat O[G]$-modules forms the stable module category $\StMod(G,\cat O)$, an integral analogue of the stable module category in modular representation theory. It inherits the structure of a tensor-triangulated category. A full subcategory  is said to be a localizing tensor ideal if it is closed under triangles, coproducts, and tensor products with arbitrary objects in $\StMod(G,\cat O)$. The main classification result proven in this paper is:

\begin{Thm*}\label{thm:main}
There is a canonical bijective correspondence between localizing tensor ideals in $\StMod(G,\cat O)$ and subsets of $\Spec^h(H^*(G;\cat O))\setminus\Spec(\cat O)$.
\end{Thm*}

Both the construction of the bijection as well as the relation to the classification of $\cat O$-linear representations of $G$ will be described in more detail below. This result is an integral version of the main theorem of \cite{BensonIyengarKrause11a} for representations over fields. As we will explain momentarily, the proof relies on methods from homotopy theory and higher algebra as well as recent advances in tensor-triangular geometry. In fact, our approach applies to the modular case as well, thus recovering the results of Benson, Iyengar, and Krause.

\subsection{Overview of results}

The starting point for this work is the construction of a suitable version of the stable module category for a finite group $G$ with coefficients in a commutative ring $R$. We take a homotopy-theoretic perspective: Let $\Fun(BG,\Perf(R))$ be the category\footnote{More precisely, in this paper we work in the context of $\infty$-categories, see also \cref{ssec:conventions}.} of local systems on the classifying space $BG$ with values in the category of perfect complexes over $R$. This category has familiar manifestations in algebra, geometry, and topology; it is equivalent to:
\begin{itemize}
    \item (Homological algebra) The bounded derived category of $R[G]$-modules with underlying perfect non-equivariant complex.
    \item (Algebraic geometry) The category of perfect complexes over the Deligne--Mumford quotient stack $[\Spec(R)/G]$ for regular $R$.
    \item (Equivariant homotopy theory) If $R$ is regular, the category of module spectra over the genuine Tate construction $\underline{R}_G$ in genuine $G$-spectra.
\end{itemize}
The category of perfect $R[G]$-modules is a full subcategory of $\Fun(BG,\Perf(R))$, strict if the order of $G$ is not invertible in $R$. Inspired by Rickard's description of the classical stable module category for field coefficients, we then define the stable module category as the Verdier quotient
\[
\stmod(G,R) = \Fun(BG,\Perf(R)) / \Perf(R[G]).
\]
This construction was inaugurated by Mathew \cite{mathew_torus,treumannmathew_reps} and further studied by A.~Krause \cite{krause_picard}, with important precursors in the stable derived category of H.~Krause \cite{krause_stablederived} or Orlov's singularity categories \cite{Orlov04}. The stable module category $\stmod(G,R)$ naturally has the structure of a tt-category. Moreover, there is an associated big (rigidly-compactly generated) tt-category 
\[
\StMod(G,R) = \Ind\stmod(G,R).
\]

Our approach to the classification of localizing tensor ideals in $\StMod(G,R)$ is based on the tt-geometric notion of stratification for a rigidly-compactly generated tt-category $\cat T$. This theory has recently been developed systematically in joint work with Heard and Sanders \cite{bhs1}, building on ideas of Balmer \cite{Balmer05a}, Balmer--Favi \cite{BalmerFavi11}, Benson, Iyengar, Krause (BIK) \cite{BensonIyengarKrause11a,BensonIyengarKrause11b}, and Stevenson \cite{Stevenson13}. Let $\cat T^c$ be the full subcategory of compact objects in $\cat T$. There is a universal support function $\Supp$ on $\cat T$ taking values in subsets of the Balmer spectrum of $\cat T^c$. It extends naturally to localizing tensor ideals and thus provides a candidate map
\[
\Supp\colon\begin{Bmatrix}
\text{Localizing tensor} \\
\text{ideal of } \cat T
\end{Bmatrix} 
\xymatrix@C=2pc{ \ar[r] &}
\begin{Bmatrix}
\text{Subsets of}  \\
\Spc(\cat T^c)
\end{Bmatrix}.
\]
Stratification formulates conditions on $\cat T$ for this map to be a bijection and thus to solve the classification problem \emph{relative} to the Balmer spectrum $\cat T^c$. 

The Balmer spectrum of the stable module category for regular $R$ can be deduced from a recent result of Lau~\cite{lau_spcdmstacks}. The augmentation $H^*(G;R) \to R$ induces a map on (homogeneous) Zariski spectra $\Spec(R) \to \Spec^h(H^*(G;R))$. Lau's theorem then implies that Balmer's comparison map \cite{Balmer10b} provides a homeomorphism 
\begin{equation}\label{eq:introlau}
\xymatrix{\Spc(\stmod(G,R)) \ar[r]^-{\cong} & \Spec^h(H^*(G;R))\setminus\Spec(R).}
\end{equation}
For $R = k$ a field of characteristic $p$ dividing the order of $G$, this result recovers the main theorem of \cite{BensonCarlsonRickard97}. In general, this map can be described in more explicit terms with the cohomological support in the sense of Benson, Carlson, Rickard \cite{BensonCarlsonRickard95, BensonCarlsonRickard96}. In conjunction with the equivalence \eqref{eq:introlau}, our main theorem (\cref{thm:stratstmod}) may be stated succinctly as follows:

\begin{thmx}\label{thm:a}
Let $G$ be a finite group and suppose $\cat O$ is a Dedekind domain of characteristic $0$, then $\StMod(G,\cat O)$ is stratified.
\end{thmx}

Besides the aforementioned classification of localizing tensor ideals, additional consequences of this theorem are that both an integral version of the tensor product formula and the generalized telescope conjecture hold in $\StMod(G,\cat O)$. Moreover, combined with Neeman's stratification result \cite{Neeman92a} for the derived category $\cat D(\cat O)$, we prove that \cref{thm:a} is equivalent to stratification of the larger category $\Rep(G,\cat O) = \Ind\Fun(BG,\Perf(\cat O))$ of ind-local systems. \Cref{thm:stratrep} asserts:

\begin{thmx}\label{thm:b}
Let $G$ be a finite group and suppose $\cat O$ is a Dedekind domain of characteristic $0$, then $\Rep(G,\cat O)$ is stratified over $\Spec^h(H^*(G;\cat O))$.
\end{thmx}

In order to connect to the category of $\cat O$-linear representations of $G$ as mentioned in the beginning, we provide an algebraic model of $\StMod(G,\cat O)$, generalizing the classical construction of the stable module category. For any commutative ring $R$, denote by $\projmod(G,R)$ the symmetric monoidal category of $R[G]$-modules whose underlying $R$-module is projective. This category carries an exact structure, obtained by pulling back the split-exact structure from $R$-modules. We prove that with this structure, $\projmod(G,R)$ is a symmetric monoidal Frobenius category. The corresponding Quillen model structure models the $\infty$-category $\StMod(G,R)$; in particular, there is a canonical symmetric monoidal equivalence on homotopy categories
\[
\xymatrix{\projmod(G,R)/\sim \ar[r]^-{\simeq} & \Ho(\StMod(G,R)).}
\]
Here, $\sim$ is the homotopy relation, which identifies two maps whenever their difference factors through a projective $R[G]$-module. With a suitable notion of localizing tensor ideals in exact categories (\cref{def:exlocalizing}), this allows a third interpretation of \cref{thm:a}, see \cref{thm:stratex}:

\begin{thmx}\label{thm:c}
Let $G$ be a finite group and suppose $\cat O$ is a Dedekind domain of characteristic $0$, then cohomological support induces a bijection
\[
\begin{Bmatrix}
\text{Non-zero localizing} \\
\text{ideals of } \cat \projmod(G,\cat O)
\end{Bmatrix} 
\simeq
\begin{Bmatrix}
\text{Subsets of} \\
\Spec^h(H^*(G;\cat O))\setminus\Spec(\cat O)
\end{Bmatrix}.
\]
\end{thmx}

There is a corresponding classification of finitely presented representations, see \cref{rem:tstex}, which we deduce directly from \eqref{eq:introlau}. Since the latter holds for any regular Noetherian commutative ring $R$, we obtain a generic classification of all finitely presented $R$-linear $G$-representations in this case.

\subsection{Overview of methods}

This work takes place in the context of stratification in tt-geometry, as systematically developed in \cite{bhs1}. It is modelled on the notion of stratification introduced and used by Benson, Iyengar, Krause (BIK) \cite{BensonIyengarKrause11a,BensonIyengarKrause11b}, but there are some distinct differences: While BIK-stratification relies on an auxiliary action of a commutative ring on the given tt-category $\cat T$, our notion of stratification is relative to the Balmer spectrum of $\cat T^c$ and thus intrinsic. This allows us to separate the computation of the spectrum from verifying the stratification \emph{condition}. As a consequence, stratification in our sense exhibits excellent permanence properties, such as a version of Zariski descent (\cite[Thm.~8.11]{Stevenson13}, \cite[Cor.~5.5]{bhs1}). 

In order to control stratification for the tt-categories of interest and to reduce it to simpler cases, we establish a version of finite \'etale descent as well as nil-descent for stratification. The former is a significant strengthening of \cite[Thm.~6.4]{bhs1} and relies on work of Balmer \cite{Balmer16b}, while the latter is inspired by a recent result of Shaul--Williamson \cite{sw_descent} on BIK-stratification. The terminology used in stating the next theorem is introduced in \cref{sec:ttgeometry}.

\begin{thmx}\label{thm:d}
Suppose $f^*\colon \cat S \to \cat T$ is a functor of rigidly-compactly generated tt-categories. Assume that one of the following two conditions is satisfied:
    \begin{enumerate}
        \item (\'Etale descent) $f^*$ is a conservative finite \'etale tt-functor.
        \item (Nil-descent) $f^*$, its right adjoint $f_*$, as well as the further right adjoint $f^!$ are conservative, and $f^*$ induces a bijection on Balmer spectra.
    \end{enumerate}
If $\cat T$ is stratified, then so is $\cat S$.\footnote{A precursor of this result was found during a joint project with Castellana, Heard, and Sanders. A more comprehensive treatment will be part of forthcoming work with them together with Naumann and Pol.}
\end{thmx}

We apply \'etale descent for stratification to reduce from arbitrary finite groups to elementary abelian groups, i.e., groups of the form $(\bbZ/p)^{\times r}$ for some prime $p$ and rank $r\ge 1$. In more detail, restriction induces a finite \'etale tt-functor
\[
\xymatrix{\res\colon \StMod(G,R) \ar[r] & \prod_{E \subseteq G}\StMod(E,R),}
\]
where $E$ ranges through the elementary abelian subgroups of $G$. We then prove a derived version of Chouinard's theorem \cite{Chouinard76} that says that $\res$ is conservative; here, the main ingredient is a theorem of Carlson \cite{Carlson00}. We may thus apply \'etale descent to reduce to the case of elementary abelian groups. This reduction step is analogous to the strategy employed in \cite{BensonIyengarKrause11a}, but does not make use of any form of Quillen stratification for integral group cohomology \cite{Quillen71} as an input. In fact, writing $\orbit_{\cat E}(G)$ for the $G$-orbit category on elementary abelian subgroups, we \emph{derive} a generalization of Quillen stratification. It takes the form of a homeomorphism
\[
\xymatrix{\colim_{\orbit_{\cat E}(G)}\Spec^h(H^*(E;R)) \ar[r]^-{\simeq} & \Spec^h(H^*(G;R)),}
\]
for any regular commutative ring $R$, see \cref{cor:quillenstratification}.

Consider then a $p$-modular system $(K,A,k)$, which consists of complete discrete valuation ring $A$ of mixed characteristic $(0,p)$ with residue field $k$ and quotient field~$K$. An important and difficult problem in modular representation theory is to decide when representations lift from $k$ to $A$. Having reduced to elementary abelian $p$-groups $E$ and using special properties of their cohomology as well as nil-descent, the next step is to prove a lifting result (\cref{cor:elabgenericlifting}) for localizing tensor ideals:

\begin{thmx}\label{thm:e}
Let $A$ be a complete DVR of mixed characteristic $(0,p)$ and $E$ an elementary abelian $p$-group, then 
the reduction map $A \to k$ induces a bijection
\begin{equation}\label{eq:introgenericlifting}
\begin{Bmatrix}
\text{Localizing tensor ideals} \\
\text{of } \cat \StMod(E,A)
\end{Bmatrix} 
\xymatrix@C=2pc{ \ar[r]^-{\sim} &}
\begin{Bmatrix}
\text{Localizing tensor ideals} \\
\text{of } \cat \StMod(E,k)
\end{Bmatrix}.
\end{equation}
\end{thmx}

Informally speaking, the isomorphism \eqref{eq:introgenericlifting} expresses a kind of generic\footnote{Here and throughout, the term `generic' is used in the sense of Hopkins~\cite{Hopkins87}: it means ``up to the equivalence relation on objects given by the tt-ideal they generate''.} modular lifting: While a given $k$-linear representation might not be liftable to $A$, it can be lifted generically, i.e., up to the tt-geometric operations available in $\StMod(G,A)$. Together with Lau's computation, \cref{thm:e} allows us to lift stratification of $\StMod(E,k)$ to stratification for $\StMod(E,A)$. The field case was proved by Benson, Iyengar, and Krause \cite{BensonIyengarKrause11a}, but we also include a self-contained homotopy-theoretic argument for the modular case, due to Mathew.  

It now remains to glue the stratifications of $\StMod(E,\cdvr{\cat O}{\frak p})$ together over the points $\frak p$ of the spectrum $\Spec(\cat O)$. To this end, there is the following arithmetic local-to-global theorem (\cref{cor:arithmeticstmodltg1}) for the stable module category:

\begin{thmx}\label{thm:f}
Let $G$ be a finite group and let $\cat O$ be a Dedekind domain in which $|G|\neq 0$. The completion maps $\cat O \to \cdvr{\cat O}{\frak p}$ induce a symmetric monoidal equivalence
\[
\xymatrix{\StMod(G,\cat O) \ar[r]^-{\sim} & \prod_{\frak p\mid |G|}\StMod(G, \cdvr{\cat O}{\frak p}),}
\]
where the product is indexed on prime ideals of $\cat O$ dividing the order of $G$.
\end{thmx}

This result is essentially due to Krause \cite{krause_picard}, who proved it for $\cat O = \bbZ$, but his argument extends easily to the case of Dedekind domains of characteristic $0$. Our approach however is different, as we deduce it from an arithmetic pullback square for $\Rep(G,\cat O)$ that includes the Archimedean place as well. 

In conclusion, \cref{thm:e} and \cref{thm:f} together imply stratification of the stable module category for elementary abelian groups $E$, which then yields \cref{thm:a} by virtue of \'etale descent. A schematic summary of the entire strategy is given in the next diagram.

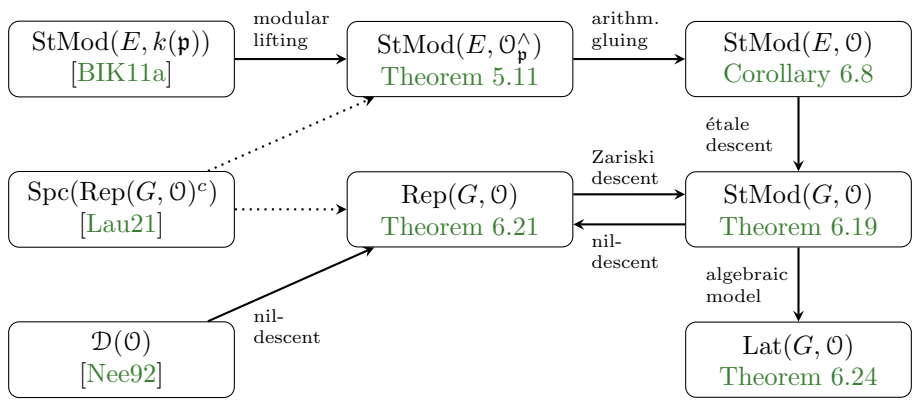
\begin{figure}[h]
\centering
\begin{tikzpicture}[node distance=2cm]
\node (S1) [category] {$\StMod(E,k(\frak p))$ \\ \cite{BensonIyengarKrause11a}};
\node (S2) [category, right of=S1, xshift=2.5cm] {$\StMod(E,\cdvr{\cat O}{\frak p})$ \\ \cref{thm:elabstmodstratification}};
\node (S3) [category, right of=S2, xshift=2.5cm] {$\StMod(E,\cat O)$ \\ \cref{cor:arithmeticstmodltg1}};

\node (L) [category, below of=S1] {$\Spc(\Rep(G,\cat O)^c)$ \\ \cite{lau_spcdmstacks}};

\node (S5) [category, below of=S2] {$\Rep(G,\cat O)$ \\ \cref{thm:stratrep}};
\node (S4) [category, below of=L] {$\cat D(\cat O)$ \\ \cite{Neeman92a}};
\node (S6) [category, below of=S3] {$\StMod(G,\cat O)$ \\ \cref{thm:stratstmod}};

\node (S7) [category, below of=S6] {$\projmod(G,\cat O)$ \\ \cref{thm:stratex}};

\draw [arrow] (S1) -- node[anchor=south] {\scriptsize\parbox{1cm}{modular lifting}} (S2);
\draw [arrow] (S2) -- node[anchor=south] {\scriptsize\parbox{1cm}{arithm. gluing}} (S3);
\draw [arrow] (S3) -- node[anchor=east] {\scriptsize\parbox{1.1cm}{\'etale \\ descent}} (S6);

\draw [arrow, transform canvas={yshift = -0.2cm}] (S6) -- node[anchor=north] {\scriptsize\parbox{1cm}{nil-descent}} (S5);
\draw [arrow, transform canvas={yshift = 0.2cm}] (S5) -- node[anchor=south] {\scriptsize\parbox{1cm}{Zariski descent}}  (S6);

\draw [arrow] (S6) -- node[anchor=east] {{\scriptsize\parbox{1.1cm}{algebraic model}}} (S7);

\draw [arrow2] (L) -- (S2);
\draw [arrow2] (L) -- (S5);
\draw [arrow] (S4) -- node[anchor=north, yshift=-0.2cm] {{\scriptsize\parbox{1cm}{nil-descent}}} (S5);

\end{tikzpicture}
\caption{Leitfaden.}
\end{figure}

\subsection{Conventions and terminology}\label{ssec:conventions}

Throughout this paper, we will work in the setting of Balmer's tensor-triangular geometry \cite{Balmer05a} and in particular use the language of stratification developed in \cite{bhs1}. All triangulated categories occurring in this paper have a model, so that the term ``tt-category'' usually refers either to a symmetric monoidal stable $\infty$-category or its homotopy category. Only in places were the difference might cause confusion, we will make a notational distinction. In addition, we therefore (implicitly or explicitly) use Lurie's theory of higher algebra \cite{HTTLurie,HALurie} whenever convenient. This becomes especially potent in the model-independent construction of the integral stable module category. The reader not familiar with the theory of $\infty$-categories may read entirely on the triangulated level, i.e., pass to homotopy categories. The only result that relies crucially on an enrichment is \cref{thm:arithmeticrepltg}, but it is not required to prove the main theorem.  

\subsection*{Acknowledgements}

I would like to thank Scott Balchin, Paul Balmer, Rudradip Biswas, Nat\`alia Castellana, Drew Heard, Achim Krause, Henning Krause, and Beren Sanders for discussions related to this work. Special thanks go to Paul Balmer for numerous helpful comments on an earlier draft of this manuscript. 

\section{Stratification in tt-geometry and descent}\label{sec:ttgeometry}

Our main results are formulated in the language of stratification in tt-geometry, which was developed in \cite{bhs1} and unifies previous work of Balmer--Favi \cite{BalmerFavi11} and Stevenson \cite{Stevenson13} as well as Benson, Iyengar, and Krause \cite{BensonIyengarKrause08,BensonIyengarKrause11b}. We summarize the basics of the theory in \cref{ssec:stratificationbasics}. Our approach \emph{separates} the problem of classifying localizing tensor ideals of a given tt-category $\cat T$ into two parts: 
\begin{enumerate}
    \item Establish stratification for $\cat T$ over its Balmer spectrum $\Spc(\cat T^c)$.
    \item Compute the Balmer spectrum of $\cat T^c$. 
\end{enumerate}
This is in sharp contrast to the strategy of Benson, Iyengar, and Krause, which uses an auxiliary action of a ring $R$ on $\cat T$ to try to settle both problems \emph{simultaneously}. One of the advantages of the separation lies in the excellent permanence properties of stratification in our sense, expressed through three types of descent (see \cref{ssec:stratificationdescent}).

\subsection{Stratification over the Balmer spectrum}\label{ssec:stratificationbasics}

We begin with a rapid introduction to the basics of tt-geometry \cite{Balmer05a}, via the perspective developed in \cite{kockpitsch17}. Let $\cat K$ be a small rigid tt-category. The operations $\oplus$ and $\otimes$ equip the collection of principal thick tensor ideals of $\cat K$ with the structure of a distributive lattice, which under Stone duality corresponds to a spectral space. The Balmer spectrum $\Spc(\cat K)$ of $\cat K$ is then defined as the Hochster dual of this space. This construction extends to a functor: Any tt-functor $f\colon \cat K \to \cat L$ gives a spectral map on Balmer spectra via pullback. Moreover, by \cite[Prop.~3.13]{Balmer05a}, the map $\cat K \to \cat K^{\natural}$ from a tt-category to its idempotent completion induces a homeomorphism on spectra. The assignment $\cat K \mapsto \Spc(\cat K)$ is analogous to the passage from a commutative ring to its Zariski spectrum. We refer the reader to \cite{Balmer05a, Balmer10b} for more details on this construction and its key properties. 

The Balmer spectrum affords a notion of support $\supp$ for objects in $\cat K$ which is compatible with the tt-structure of $\cat K$. In particular, it induces a map from thick tensor ideals of $\cat K$ to subsets of $\Spc(\cat K)$. The support of any object is closed with quasi-compact complement, so the support of any thick tensor ideals is Thomason (aka dual-open), i.e., it can be written as a union of closed subsets with quasi-compact complement. Essentially by construction, this map induces a bijection
\begin{equation}\label{eq:balmerclassification}
\supp\colon\begin{Bmatrix}
\text{Thick tensor ideals} \\
\text{of } \cat K
\end{Bmatrix} 
\xymatrix@C=2pc{ \ar[r]^-{\sim} &}
\begin{Bmatrix}
\text{Thomason subsets of}  \\
\Spc(\cat K)
\end{Bmatrix}.
\end{equation}
In fact, the fundamental theorem of tt-geometry \cite[Thm.~3.2 and Thm.~5.2]{Balmer05a} states that the pair $(\Spc(\cat K),\supp)$ is universal with this property. 

Switching now to suitable big tt-categories, Balmer and Favi \cite{BalmerFavi11} introduced an extension of the support function from $\cat T^c$ to $\cat T$ which will be instrumental for our theory of stratification. For simplicity of exposition, we will impose the following finiteness condition throughout this paper; see also \cref{rem:weaklynoetherian}.

\begin{Conv}
Throughout the remainder of this section, we will assume that the small tt-categories under consideration have Noetherian Balmer spectrum. 
\end{Conv}

All concrete examples of tt-categories appearing in this paper will satisfy this property. One useful consequence is that a subset of the Balmer spectrum is Thomason if and only if it is specialization closed. We also note that Noetherian spaces have good permanence properties: Every subspace as well as target of a surjective continuous map of a Noetherian space is Noetherian (see \cite[Cor.~8.1.6]{DickmannSchwartzTressl19}). 

\begin{Rec}\label{rec:indcat}
Let $\cat T$ be a rigidly-compactly generated tt-category, i.e., a compactly generated tt-category in which the dualizable objects coincide with the compact ones. We write $\cat T^c$ for the full subcategory of compact objects in $\cat T$. Conversely, if $\cat K$ is a rigid tt-category which admits an $\infty$-categorical enhancement, we can pass to the corresponding ind-category $\Ind(\cat K)$, which has the effect of formally adding filtered colimits to $\cat K$. The $\infty$-category $\Ind(\cat K)$ inherits the structure of a compactly generated tt-category by \cite[\S5.3.5]{HTTLurie} \cite[Prop.~1.1.3.6 and Cor.~4.8.1.14]{HALurie}. In fact, the operations $\Ind$ and $(-)^c$ are almost inverse to each other \cite[Prop.~5.5.7.8]{HTTLurie}: While $\cat T = \Ind(\cat T^c)$, the canonical map $\cat K \to \Ind(\cat K)^c$ exhibits the target as the idempotent completion $\cat K^{\natural}$ of $\cat K$.
\end{Rec}

Balmer and Favi's construction proceeds as follows: Let $\cat T$ be a rigidly-compactly generated tt-category. First, any Thomason subset $\cat Y \subseteq \Spc(\cat T^c)$ yields a triangle of idempotents
\[
\xymatrix{e_{\cat Y} \to \unit \to f_{\cat Y}}
\]
representing colocalization and localization at the thick tensor ideal $\cat T^c_{\cat Y} = \supp^{-1}(\cat Y)$ corresponding to $\cat Y$ under \eqref{eq:balmerclassification}. Since $\Spc(\cat T^c)$ is assumed to be Noetherian, any singleton $\{x\} \subseteq \Spc(\cat T^c)$ can be written as an intersection $\{x\} = \cat V \cap \cat W^c$ for Thomason subsets $\cat V, \cat W \subseteq \Spc(\cat T^c)$. Define idempotents $\Gamma_{x}\unit = e_{\cat V} \otimes f_{\cat W}$; this is independent of the choices made. We can extend this to functors $\Gamma_x,\Lambda^x\colon \cat T \to \cat T$ by 
\[
\Gamma_xt = t \otimes \Gamma_x\unit \qquad \text{and} \qquad \Lambda^xt = \iHom(\Gamma_x\unit, t);
\]
by construction, $\Gamma_x$ is left adjoint to $\Lambda^x$. With these definitions at hand, we are ready to construct support and cosupport functions for big objects:

\begin{Def}
Let $t \in \cat T$. The \emph{support} and \emph{cosupport} of $t$ are given by
\[
\Supp(t) = \SET{x \in \Spc(\cat T^c)}{\Gamma_{x}t \neq 0} \quad \text{and} \quad \Cosupp(t) = \SET{x \in \Spc(\cat T^c)}{\Lambda^{x}t \neq 0}.
\]
\end{Def}

\begin{Rem}
When $t \in \cat T$ is compact, then $\Supp(t) = \supp(t)$. However, in general the support may differ significantly from the cosupport even for compact objects (see for example \cite[Ex.~11.1]{BIK_cosupport} or \cite{bchs1}). 
\end{Rem}

The Balmer--Favi notion of supports extends naturally to localizing tensor ideals, by setting $\Supp(\cat L) = \bigcup_{l\in\cat L}\Supp(l)$ for any localizing tensor ideal $\cat L \subseteq \cat T$: 
\begin{equation}\label{eq:lococlassification}
\Supp\colon\begin{Bmatrix}
\text{Localizing tensor ideals} \\
\text{of } \cat T
\end{Bmatrix} 
\xymatrix@C=2pc{ \ar[r] &}
\begin{Bmatrix}
\text{Subsets}  \\
\text{of } \Spc(\cat T^c)
\end{Bmatrix}.
\end{equation}
We say that $\cat T$ is \emph{stratified} (over $\Spc(\cat T^c)$) if the map \eqref{eq:lococlassification} is a bijection. We note that $\Supp$ is surjective if $\Spc(\cat T^c)$ is Noetherian, so stratification is really about injectivity. This observation leads to the following characterization \cite[Thm.~4.1]{bhs1}:

\begin{Prop}
A rigidly-compactly generated tt-category $\cat T$ is stratified if the following two conditions hold:
    \begin{enumerate}
        \item (Local-to-global) For every $t \in \cat T$, we have an equality of ideals
            \[ 
            \Loco{t} = \Loco{\Gamma_xt \mid x \in \Supp(t)}.
            \]
        \item (Minimality) For each $x \in \Spc(\cat T^c)$, the localizing tensor ideal $\Gamma_x\cat T$ does not contain any non-zero proper localizing subideals. 
    \end{enumerate}
\end{Prop}

When $\Spc(\cat T^c)$ is Noetherian, then the local-to-global principle holds for $\cat T$ by \cite[Thm.~6.9]{Stevenson13} or \cite[Thm.~3.21]{bhs1}, so stratification reduces to the minimality of the tensor ideals $\Gamma_x\cat T$ for all $x \in \Spc(\cat T^c)$. There is a criterion for minimality of localizing tensor ideals, which is a mild variation of a result originally due to Benson, Iyengar, and Krause \cite[Lem.~3.9]{BensonIyengarKrause11a}:

\begin{Lem}\label{lem:minimality_crit}
Let $\cat T$ be a rigidly-compactly generated tt-category and $\cat L$ a nonzero localizing tensor ideal of $\cat T$. The following statements are equivalent:
	\begin{enumerate}
		\item The localizing tensor ideal $\cat L$ is minimal. 
		\item For all non-zero $t_1,t_2 \in \cat L$ there exists $z \in \cat T$ such that $\Hom_{\cat T}(t_1 \otimes z,t_2) \ne 0$. 
	\end{enumerate}
\end{Lem}

An important consequence of stratification of a given $\cat T$ with Noetherian spectrum is the generalized telescope conjecture in $\cat T$; this result is due to Stevenson \cite[Thm.~7.15]{Stevenson13} and was extended to generically Noetherian spectra in \cite[Thm.~9.11]{bhs1}.

\begin{Prop}[Stevenson]\label{prop:gentelescope}
If $\cat T$ is stratified with Noetherian spectrum $\Spc(\cat T^c)$, then the telecope conjecture holds in $\cat T$, i.e., there are bijections
\[
\begin{Bmatrix}
\text{Smashing ideals} \\
\text{of } \cat T
\end{Bmatrix} 
\xymatrix@C=1.2pc{ \ar[r]_-{(-)^c}^-{\sim} &}
\begin{Bmatrix}
\text{Thick tensor ideals} \\
\text{of } \cat T^c
\end{Bmatrix} 
\xymatrix@C=2pc{ \ar[r]_-{\supp}^-{\sim} &}
\begin{Bmatrix}
\text{Specialization closed}  \\
\text{subsets of } \Spc(\cat T^c)
\end{Bmatrix}.
\]
\end{Prop}

A further consequence of stratification is that the universal notion of support satisfies the tensor product formula, see \cite[Thm.~8.2]{bhs1}. This terminology is inspired by the example of modular representation theory.

\begin{Prop}\label{prop:tensorproductformula}
If $\cat T$ is stratified with Noetherian spectrum $\Spc(\cat T^c)$, then for any two $t_1,t_2 \in \cat T$ we have an equality $\Supp(t_1 \otimes t_2) = \Supp(t_1) \cap \Supp(t_2)$.
\end{Prop}

\begin{Rem}\label{rem:weaklynoetherian}
With the exception of \cref{prop:gentelescope}, the results in this subsection hold more generally for any rigidly-compactly generated tt-category with weakly Noetherian Balmer spectrum. A spectral space is called weakly Noetherian if every singleton can be written as an intersection of a Thomason subset and the complement of a Thomason subset. We refer to \cite{bhs1} for the details. 
\end{Rem}

\subsection{Descent for stratification}\label{ssec:stratificationdescent}

Stratification of tt-categories has a number of strong permanence properties which we will present in this section. The first is taken from \cite{bhs1}, while the second and third are new to this paper. A more comprehensive treatment will be given in forthcoming joint work with  Castellana, Heard, Naumann, Pol, and Sanders. 

\subsubsection{Zariski descent}

Let $V$ be the complement of a Thomason subset $\cat Y \subseteq \Spc(\cat T^c)$. Then there exists a finite localization 
\[
\xymatrix{\cat T \ar[r] & \cat T(V) = \cat T/\Loco{\SET{t \in \cat T^c}{\supp(t) \subseteq \cat Y}}}
\]
which on spectra induces the inclusion $V \hookrightarrow \Spc(\cat T^c)$. In \cite[Cor.~5.5]{bhs1}, we have established the following form of Zariski descent for stratification:

\begin{Thm}\label{thm:zariskidescent}
Let $\Spc(\cat T^c) = \bigcup_{i\in i}V_i$ be a cover by complements of Thomason subsets, then $\cat T$ is stratified if and only if $\cat T(V_i)$ is stratified for all $i\in I$.
\end{Thm}

\begin{Cor}\label{cor:stratfiniteproduct}
Let $(\cat T_i)_{i \in I}$ be a finite collection of rigidly-compactly generated tt-categories, then $\prod_i\cat T_i$ is stratified if and only if the tt-categories $\cat T_i$ are stratified for all $i$.
\end{Cor}
\begin{proof}
This can be checked by hand; alternatively, we may argue as follows: Since the collection is finite, the projection functors induce a homeomorphism of Balmer spectra
\[
\textstyle{\coprod_{i\in I} \Spc(\cat T_i^c) \cong \Spc(\prod_{i \in I}\cat T_i^c).}
\]
In particular, for any $j\in I$, the projection functor $\prod_{i \in I}\cat T_i \to \cat T_j$ identifies with the finite localization corresponding to the specialization closed subset $\coprod_{i \neq j} \Spc(\cat T_i^c) \subseteq \Spc(\prod_{i \in I}\cat T_i^c)$. The claim therefore follows from Zariski descent, \cref{thm:zariskidescent}.
\end{proof}

\subsubsection{\'Etale descent}

Stratification of tt-categories also satisfies a version of \'etale descent, as is proven in \cite[Thm.~6.4]{bhs1}. Here, we prove a strengthening of this result which removes the hypothesis on the fiber. In order to state it, we recall that a tt-functor $f^*\colon \cat S \to \cat T$ between rigidly compactly generated tt-categories is said to be finite \'etale if it satisfies the following properties:
    \begin{enumerate}
        \item $f^*$ is cocontinuous, i.e., preserves set-indexed coproducts;
        \item there exists a compact commutative separable algebra object $\cat A \in \cat T$ of finite degree (\cite{Balmer14}) and an equivalence of tt-categories $\psi\colon\cat T \simeq \Mod_{\cat S}(\cat A)$ such that the following diagram commutes:
            \[
            \xymatrix{\cat S \ar[r]^-{f^*} \ar[rd]_{-\otimes A} & \cat T \ar[d]_{\wr}^{\psi} \\
            & \Mod_{\cat S}(\cat A).}
            \]
    \end{enumerate}
In particular, this implies that the right adjoint $f_*$ to $f^*$ is conservative. Our \'etale descent theorem for stratification is:

\begin{Thm}\label{thm:etaledescent}
Let $f^*\colon \cat S \to \cat T$ be a conservative finite \'etale map of rigidly-compactly generated tt-categories. If $\cat T$ is stratified, then so is $\cat S$.
\end{Thm}

In order to prepare for the proof of this theorem, we start with a point-set topological observation.

\begin{Lem}\label{lem:closedfiber}
Let $f^*\colon \cat S \to \cat T$ be a finite \'etale extension of rigidly-compactly tt-categories and write $\varphi\colon \Spc(\cat T^c) \to \Spc(\cat S^c)$ for the induced map on spectra. If $\cat S$ is local with closed point $\frak m$, then the fiber $\varphi^{-1}(\frak m) = \{\cat P_1,\ldots, \cat P_d\}$ is finite discrete and consists of Thomason points in $\Spc(\cat T^c)$.
\end{Lem}
\begin{proof}
Since $\Spc(\cat T^c)$ is Noetherian, a point in it is closed if and only if the corresponding singleton is Thomason. By \cite[Thm.~1.5(d)]{Balmer16b}, the fibers of $\varphi$ are finite and discrete. Consider $\cat P \in \varphi^{-1}(\frak m)$ and $\cat Q \subseteq \cat P$. Since $\varphi$ is inclusion preserving and $\frak m$ is closed, $\varphi(\cat Q) \subseteq \frak m$ is an equality, so $\cat Q = \cat P$ by discreteness of the fiber. It follows that all points in $\varphi^{-1}(\frak m)$ are closed.
\end{proof}

\begin{Not}
Consider a tt-category  $\cat K$ with Noetherian Balmer spectrum. If $\cat X$ is a subset of $\Spc(\cat K)$ with $\cat X = \cat V \cap \cat W^c$ for Thomason subsets $\cat V, \cat W \subseteq \Spc(\cat K)$, then we set 
\[
\kappaaux_{\cat K}(\cat X) = e_{\cat V} \otimes f_{\cat W}.
\]
If $\cat K$ arises as the full subcategory of compact objects in a rigidly-compactly tt-category $\cat T^c$, we will also use the subscript $\cat T$ in place of $\cat T^c$ on $g$. For a singleton $\cat Y = \{x\}$, we recover the previous construction $\kappaaux_{\cat K}(x) := \kappaaux_{\cat K}(\{x\}) \simeq \Gamma_x\unit$. 
\end{Not}

\begin{Lem}\label{lem:ressupp}
Suppose $f^*\colon \cat S \to \cat T$ is a tt-functor which admits a conservative right adjoint $f_*$, and write $\varphi\colon \Spc(\cat T^c) \to \Spc(\cat S^c)$ for the map induced by $f^*$ on spectra. For any $t \in \cat T$, there is an equality $\Supp(f_*(t)) = \varphi(\Supp(t))$.
\end{Lem}
\begin{proof}
Consider a prime $\cat Q \in \Spc(\cat S^c)$. Choose Thomason subsets $\cat V, \cat W \subseteq \Spc(\cat S^c)$ with ${\cat Q} = \cat V \cap \cat W^c$. By \cite[Thm.~6.3]{BalmerFavi11}, we have equivalences
\[
f^*(\kappaSQ) \simeq f^*(e_{\cat V}) \otimes f^*(f_{\cat W}) \simeq e_{\varphi^{-1}\cat V} \otimes f_{\varphi^{-1}\cat W} \simeq  g_{\cat T}(\varphi^{-1}(\cat Q)).
\]
For any $t \in \cat T$, the projection formula then provides equivalences
\[
\kappaSQ \otimes f_*(t) \simeq f_*(f^*(\kappaSQ) \otimes t) \simeq f_*(g_{\cat T}(\varphi^{-1}(\cat Q)) \otimes t).
\]
Since $f_*$ is conservative, $\kappaSQ \otimes f_*(t) = 0$ if and only if $g_{\cat T}(\varphi^{-1}(\cat Q)) \otimes t = 0$. In support-theoretic terms, this is equivalent to the statement that $\cat Q \in \Supp(f_*(t))$ if and only if $\varphi^{-1}(\cat Q) \cap \Supp(t) \neq \varnothing$, as desired.
\end{proof}

\begin{Lem}\label{lem:indfullsupp}
With notation as in \cref{lem:closedfiber}, let $s \in \cat S$ be an object with $\Supp(s) = \{\frak m\}$. Then $\Supp(f^*(s)) = \varphi^{-1}(\frak m)$.
\end{Lem}
\begin{proof}
By assumption, $s \in \Loco{e_{\{\frak m\}}}$, hence $f^*(s) \in \Loco{f^*(e_{\{\frak m\}})}$. We know from \cite[Thm.~6.3]{BalmerFavi11} that $f^*(e_{\{\frak m\}}) = e_{\varphi^{-1}(\frak m)}$, so $f^*(s) \in \Loco{e_{\varphi^{-1}(\frak m)}}$. This shows that $\Supp(f^*(s)) \subseteq \varphi^{-1}(\frak m)$.

It remains to verify the inclusion $\varphi^{-1}(\frak m) \subseteq \Supp(f^*(s))$. To this end, consider $\cat P \in \varphi^{-1}(\frak m)$. By \cref{lem:closedfiber}, $\cat P$ is a Thomason point in $\Spc(\cat T^c)$. Therefore, there exists $\kappa(\cat P) \in \cat T^c$ with $\supp(\kappa(\cat P)) = \{\cat P\}$. Let $f_*\colon \cat T \to \cat S$ be the right adjoint to $f^*$. Suppose $\cat P \notin \Supp(f^*(s))$, then $\kappa(\cat P) \otimes f^*(s) = 0$, hence 
\[
0 = f_*(\kappa(\cat P) \otimes f^*(s)) \simeq f_*(\kappa(\cat P)) \otimes s.
\]
Since $f^*$ is finite, it follows that $f_*$ preserves compact objects, so we may apply the half-tensor product theorem of Balmer--Favi \cite[Thm.~7.22]{BalmerFavi11} to compute 
\[
\varnothing = \Supp(f_*(\kappa(\cat P)) \otimes s) = \supp(f_*(\kappa(\cat P))) \cap \Supp(s).
\]
We see from \cref{lem:ressupp} that $\supp(f_*(\kappa(\cat P))) = \{\varphi(\cat P)\} = \{\frak m\}$, which implies $\frak m \notin \Supp(s)$, a contradiction. The claim follows.
\end{proof}

The rest of the argument proceeds along similar lines as the proof of \cite[Thm.~6.4]{bhs1}.

\begin{Lem}\label{lem:closedpointcriterion}
With notation as in \cref{lem:closedfiber}, assume additionally that $\cat T$ has minimality at all points of $\varphi^{-1}(\frak m)$. Then minimality holds for $\cat S$ at $\frak m$.
\end{Lem}
\begin{proof}
In order to establish minimality of $\cat S$ at $\frak m$, we need to prove that $\Loco{s} = \Gamma_{\frak m}\cat T$ for any $s \in \cat S$ with $\Supp(s) = \{\frak m\}$. Equivalently, because $\frak m$ is a closed point, we have to show that $e_{\frak m} \in \Loco{s}$. 

Consider an object $t$ with $\Supp(s) = \frak m$ and let $\cat P \in \varphi^{-1}(\frak m)$. \Cref{lem:closedfiber} says that $\cat P$ is Thomason, so in particular $e_{\{\cat P\}} \simeq \kappaSP$. By \cref{lem:indfullsupp} and definition of support, $\kappaSP \otimes f^*(s) \neq 0$, so minimality at $\cat P$ implies
\begin{equation}\label{eq:closedpointcriterion}
\kappaSP \in \Loco{\kappaSP \otimes f^*(s)} \subseteq \Loco{f^*(s)}.
\end{equation}
The support of $e_{\varphi^{-1}(\frak m)}$ is precisely $\varphi^{-1}(\frak m)$, and it follows from the local-to-global principle in $\cat T$ together with \cref{eq:closedpointcriterion} that $f^*(e_{\{\frak m\}}) \simeq e_{\varphi^{-1}(\frak m)} \in \Loco{f^*(s)}$. 

As in the proof of \cite[Thm.~6.4]{bhs1} or directly via \cite[Lem.~3.2]{Balmer16b}, the localizing subcategory $f_*^{-1}\Loco{f_*f^*(s)}$ is in fact a localizing tensor ideal in $\cat T$. Since $f^*(s) \in f_*^{-1}\Loco{f_*f^*(s)}$, we have
\[
f^*(e_{\{\frak m\}}) \in \Loco{f^*(s)} \subseteq f_*^{-1}\Loco{f_*f^*(s)}.
\]
The projection formula shows $f_*f^*(s') \simeq f_*(f^*(\unit) \otimes f^*(s')) \simeq f_*f^*(\unit) \otimes s'$ for any $s' \in \cat S$, so this implies:
\[
\Loco{f_*f^*(e_{\{\frak m\}})} \subseteq \Loco{f_*f^*(s)} \subseteq \Loco{s}.
\]
Since $f_*\unit$ is compact and non-zero, $\frak m \in \supp(f_*\unit)$, hence $\Supp(e_{\{\frak m\}}) \subseteq \supp(f_*\unit)$. From \cite[Lem.~3.7]{bhs1}, we thus get $e_{\{\frak m\}} \in \Loco{e_{\{\frak m\}} \otimes f_*\unit} = \Loco{f_*f^*(e_{\{\frak m\}})}$, so $e_{\{\frak m\}} \in \Loco{s}$. This establishes minimality of $\cat S$ at $\frak m$.
\end{proof}

\begin{Lem}\label{lem:generalpointcriterion}
In the situation of \cref{thm:etaledescent}, consider a prime $\cat P \in \Spc(\cat T^c)$. If $\cat T$ has minimality at all points of $\varphi^{-1}(\varphi(\cat P))$, then $\cat S$ has minimality at $\varphi(\cat P)$.
\end{Lem}
\begin{proof}
The finite localization away from the prime ideal $\varphi(\cat P)$ induces a commutative diagram
\[
\xymatrix{\cat S \ar[r]^-{f^*} \ar[d] & \cat T \ar[d] \\
\cat S/\langle \varphi(\cat P)\rangle \ar[r]_-{g^*} & \cat T(W),}
\]
where $W= \varphi^{-1}(\gen(\varphi(\cat P)))$, see \cite[Prop.~1.30]{bhs1}. The two vertical maps are finite localizations and $\cat S/\langle \varphi(\cat P)\rangle = \cat S(\gen(\varphi(\cat P)))$ is local with closed point $\frak m = \varphi(\cat P)$. It follows that $\cat S$ is minimal at $\cat P$ if and only if $\cat S/\langle \varphi(\cat P)\rangle$ is minimal at its closed point $\frak m$. Moreover, the induced functor $g^*$ is finite \'etale by \cite[Ex.~5.12]{Sanders21pp}. By \cref{lem:closedpointcriterion}, $\cat S/\langle \varphi(\cat P)\rangle$ is minimal at $\frak m$ if $\cat T(W)$ has minimality at all points of $\varphi^{-1}(\frak m)$. Minimality at these points is detected in $\cat T$, so we are done.
\end{proof}

\begin{Rem}\label{rem:surjectivity}
If $f^*\colon \cat S \to \cat T$ is a conservative finite \'etale morphism of rigidly compactly generated tt-categories, then the induced map on spectra is surjective. Indeed, we can check this locally in $\cat S$, so we can appeal to \cite[Thm.~1.2]{Balmer18}, for example. In particular, if $\cat T$ is Noetherian, then so is $\cat S$. In this case, the local-to-global principle holds for $\cat S$.
\end{Rem}

\begin{proof}[Proof of \cref{thm:etaledescent}]
The Balmer spectrum of the tt-category $\cat S$ is Noetherian and hence $\cat S$ satisfies the local-to-global principle. As observed in \cref{rem:surjectivity}, the induced map $\varphi$ on spectra is surjective. Let $\cat Q \in \Spc(\cat S^c)$ be a prime, so there exists some $\cat P \in \Spc(\cat T^c)$ with $\varphi(\cat P) = \cat Q$. Minimality of $\cat S$ at $\cat Q$ then follows from \cref{lem:generalpointcriterion} and assumption on $\cat T$.
\end{proof}

\subsubsection{Nil-descent}

There is a third kind of descent which we will use repeatedly in this paper. It exhibits conditions for descent of stratification along a  $tt$-functor $f^*\colon \cat S \to \cat T$ which has the property that the induced map on spectra $\varphi\colon \Spc(\cat T^c) \to \Spc(\cat S^c)$ is a bijection. Informally speaking, one may thus think of $\cat T$ as a nilpotent thickening of $\cat S$, so that we refer to it as \emph{nil-descent}. We will need a base-change lemma for support and cosupport from \cite{bchs1}:

\begin{Lem}\label{lem:cosuppbasechange}
Let $f^*\colon \cat S \to \cat T$ be a tt-functor with right adjoint $f_*$, which admits a further right adjoint $f^!$. We write $\varphi$ for the map induced by $f^*$ on Balmer spectra. Assume that $f_*$ is conservative and let $s \in \cat S$, then:
    \begin{enumerate}
        \item $\varphi(\Supp(f^*s)) \subseteq  \Supp(s)$, with equality if $f^*$ is conservative.
        \item $\varphi(\Cosupp(f^!s)) \subseteq  \Cosupp(s)$, with equality if $f^!$ is conservative.
    \end{enumerate}
\end{Lem}

\begin{Thm}\label{thm:nildescent}
Let $(f^*,f_*,f^!)\colon \cat S \to \cat T$ be a triple of adjoints with $f^*$ a tt-functor and suppose $\Spc(\cat T^c)$ is Noetherian. Assume that
    \begin{enumerate}
        \item all of $f^*$, $f_*$, and $f^!$ are conservative;
        \item $\varphi\colon \Spc(\cat T^c) \to \Spc(\cat S^c)$ is bijective.\footnote{The map $\varphi$ is always spectral. Note, however, that a bijective spectral map is not necessarily a homeomorphism.}
    \end{enumerate}
If $\cat T$ is stratified, then so is $\cat S$. 
\end{Thm}
\begin{proof}
Because $\Spc(\cat T^c)$ is Noetherian, it follows from the second assumption that $\Spc(\cat S^c)$ is Noetherian as well. As noted before, this implies that the local-to-global principle holds for $\cat S$. It thus remains to verify minimality at all primes $\cat P \in \Spc(\cat S^c)$.  

Let $s_1,s_2 \in \Gamma_{\cat P}\cat S$ be non-zero objects. By adjunction, $\cat P \in \Cosupp(s_2)$. Under our assumptions, the base-change formulas of \cref{lem:cosuppbasechange} show that 
\begin{align*}
    \Supp(f^*s_1) & = \varphi^{-1}\Supp(s_1) = \{\varphi^{-1}(\cat P)\} \\
    \Cosupp(f^!s_2) & = \varphi^{-1}\Cosupp(s_2) \ni \varphi^{-1}(\cat P).
\end{align*}
Since $\cat T$ is stratified, the cosupport of the internal Hom in $\cat T$ for any two objects $t_1,t_2 \in \cat T$ is given by $\Cosupp\iHom(t_1,t_2) = \Supp(t_1) \cap \Cosupp(t_2)$, see \cite[Thm.~4.1]{bchs1}. In particular, we get:
\[
\Cosupp\iHom(f^*s_1,f^!s_2) = \Supp(f^*s_1) \cap \Cosupp(f^!s_2) \ni \varphi^{-1}(\cat P).
\]
It follows that there exists $z \in \cat T$ such that 
\[
0 \neq \Hom_{\cat T}(z,\iHom(f^*s_1,f^!s_2)) \simeq \Hom_{\cat T}(z \otimes f^*s_1,f^!s_2).
\]
Adjunction and the projection formula imply that $\Hom_{\cat S}((f_*z) \otimes s_1,s_2)$ is non-trivial as well, so we conclude that $\Gamma_{\cat P}\cat S$ is minimal by virtue of the minimality criterion \cref{lem:minimality_crit}. 
\end{proof}

\begin{Rem}
In the context of BIK-stratification, Shaul--Williamson \cite{sw_descent} prove a related result, which in fact was one source of inspiration for \cref{thm:nildescent}. However, their arguments are different from ours and in particular do not make use of cosupport.
\end{Rem}

\section{The integral stable module category}

Throughout this section, $G$ is a finite group and $R$ is a Noetherian commutative ring, usually assumed to be regular. We are interested in $R$-linear representations of $G$ whose underlying $R$-module is projective. In the derived setting and if $R$ is regular, this is not restrictive, as $R$-modules admit projective resolutions of finite length. Our main objective is to construct an appropriate integral version of the stable module category and to provide both homotopical and algebraic models for it. In the last subsection, we recall Lau's theorem on the Balmer spectrum of the Deligne--Mumford stack $[\Spec(R)/G]$ for regular $R$ and use it to compute $\Spc(\stmod(G,R))$.

\subsection{Abelian and derived categories of representations}

For simplicity, we only consider coefficient rings $R$ with trivial $G$-action; otherwise, we would have to work with twisted group rings. Write $R[G]$ for the group ring of $G$ with coefficients in $R$ and $\Mod(R[G])$ for the abelian category of $R$-linear $G$-representations, i.e., modules over $R[G]$. There is a symmetric monoidal structure $\otimes_R$ on $\Mod(R[G])$ with action obtained from the diagonal map $G \to G \times G$; the monoidal unit is the trivial module $R$. The category $\Mod(R[G])$ is Grothendieck abelian, so we can consider the (unbounded) derived category $\cat D(R[G])$. Let $\projmod(G,R)$ denote the full subcategory of $\Mod(R[G])$ on those modules which are projective as $R$-modules.

\begin{Rec}\label{rec:localsystems}
Let $\cat C = (\cat C,\otimes, \unit)$ be a compactly generarted symmetric monoidal $\infty$-category and view the classifying space $BG$ as an $\infty$-groupoid. The category of \emph{local systems} on $BG$ with coefficients $\cat C$ is defined as $\Fun(BG,\cat C)$. Informally speaking, a local system consists of an object in $\cat C$ equipped with a coherent action by $G$. Equipped with the point-wise tensor product, this has the structure of a compactly generated symmetric monoidal stable $\infty$-category with unit $\unit$. More details about the $\infty$-category of local systems can be found for example in \cite[\S4.4]{HopkinsLurie13pp} and \cite[\S3]{HALurie}.
\end{Rec}

Specializing this construction to $\cat C = \cat D(R)$, we note that $\Fun(BG,\cat D(R))$ is usually not rigidly-compactly generated: Indeed, the unit is in general not compact. The full subcategory of dualizable objects can be identified as
\[
\Fun(BG,\cat D(R))^{\dual} \simeq \Fun(BG,\Perf(R)),
\]
which contains $\Fun(BG,\cat D(R))^{\omega}$ as a full subcategory.

\begin{Def}
We define the \emph{derived category of $G$-representations} with coefficients in $R$ as the small stable $\infty$-category
\[
\rep(G,R) = \Fun(BG,\Perf(R)).
\]
The point-wise tensor product in $\Perf(R)$ equips $\rep(G,R)$ with the structure of a symmetric monoidal $\infty$-category whose tensor product $\otimes = \otimes_R$ is exact in both variables. The monoidal unit is $R = R_G$ with trivial $G$-action.
\end{Def}

\begin{Rem}\label{rem:repmodels}
Viewed as a full subcategory of $\Fun(BG,\cat D(R)) \simeq \cat D(R[G)])$
equipped with the pointwise symmetric monoidal structure, the category $\rep(G,R)$ can be identified as the bounded derived category of objects whose underlying complex of $R$-modules is perfect. If $R$ is regular, this in turn is equivalent to the full subcategory of complexes all of whose modules are finitely generated projective $R$-modules. This is also equivalent to the category of perfect complexes on the stack $[\Spec(R)/G]$ as considered by Lau, see \cite[Prop.~3.1]{lau_spcdmstacks} and the discussion around it. 
\end{Rem}

\begin{Lem}
For $M \in \projmod(G,R)$ viewed as an object of $\rep(G,R)$ concentrated in degree $0$, we have an isomorphism
\[
\pi_*M := \pi_*\Hom_{\rep(G,R)}(R,M) \cong H^{-*}(G;M).
\]
In particular, $\pi_*R$ is the group cohomology of $G$ with coefficients in $R$.
\end{Lem}

The category $\rep(G,R)$ has also appeared in Mathew's appendix to \cite{treumannmathew_reps}, where it was denoted $\mathrm{Rep}(G,R)$. There is a corresponding big version, obtained by passing to the ind-category.

\begin{Def}\label{def:Rep}
Write $\Rep(G,R) = \Ind\rep(G,R)$ for the ind-category of $\rep(G,R)$, equipped with the induced symmetric monoidal structure. 
\end{Def}

By construction, $\Rep(G,R)$ is a compactly generated tt-category with unit $R$ and with full subcategory of compact objects equivalent to the idempotent completion of $\rep(G,R)$, see \cref{rec:indcat}. 

\begin{Exa}
If $R =k$ is a field of characteristic $p$, the category $\Rep(G,k)$ is equivalent to $K(\Inj(k[G]))$, the homotopy category of unbounded chain complexes of injective $k[G]$-modules, as studied in \cite{bensonkrause_kinj} and \cite{BensonIyengarKrause11a}. Likewise, we have symmetric monoidal equivalences
\[
\rep(G,k) \simeq K(\Inj(k[G]))^c \simeq \cat D^b(\Mod^{\fp}(k[G])),
\]
the bounded derived category of finitely presented $k[G]$-modules.
\end{Exa}

If $G$ is a finite $p$-group and $k$ of characteristic $p$, then $\Rep(G,k)$ is generated by $k$. Derived Morita theory (due to Schwede--Shipley \cite{SS03}) then furnishes an equivalence $\Rep(G,k) \simeq \Mod_{\Sp}(k^{hG})$, where $k^{hG} \simeq C^*(BG,k)$ is the commutative ring spectrum of cochains on $BG$ with coefficients in $k$. There is an equivariant generalization of this result which also holds for regular coefficient rings, essentially established by Mathew. For a more general result in this direction, see \cite{balmergallauer_permutationmodulescohomsing}.

\begin{Thm}[Mathew]\label{thm:eqmoritarep}
Let $\Sp_G$ be the symmmetric monoidal stable $\infty$-category of genuine $G$-spectra for a finite group $G$. Write $\borel{R} = F(EG_+,\infl R) \in \CAlg(\Sp_G)$ for the $G$-Borel-equivariant spectrum associated to $R$. If $R$ is regular, then there is a symmetric monoidal equivalence
\[
\Mod_{\Sp_G}(\borel{R}) \simeq \Rep(G,R).
\]
Under this equivalence, the module $R[G]$ corresponds to $\borel{R}\otimes F(G_+,S_G^0)$.
\end{Thm}
\begin{proof}
By \cite[Cor.~6.21]{MathewNaumannNoel17}, there is a fully faithful symmetric monoidal exact functor
\[
\xymatrix{\Mod_{\Sp_G}(\borel{R}) \ar[r] & \Fun(BG,\cat D(R)),}
\]
which sends $\borel{R} \otimes G/H_+$ to $R[G/H]$. On the one hand, the domain is compactly generated by $\SET{\borel{R} \otimes G/H_+}{H\subseteq G}$. On the other hand, in Theorem A.4 of \cite{treumannmathew_reps}, Mathew shows that 
\begin{equation}\label{eq:permmodulesgenerate}
\Fun(BG,\Perf(R)) = \Thick{\SET{R[G/H]}{H\subseteq G}}
\end{equation}
whenever $R$ is regular, i.e., $\rep(G,R)$ is generated by  permutation modules. It follows that $\Mod_{\Sp_G}(\borel{R})^c$ is equivalent as a symmetric monoidal $\infty$-category to $\Fun(BG,\Perf(R))$. Passing to the ind-completions then finishes the proof.
\end{proof}

\begin{Rem}
Note that $\borel{R}^H \simeq R^{hH}$ for any subgroup $H$ of $G$. 
\end{Rem}

\begin{Warn}\label{warn:generation}
In contrast to the modular case, $\rep(G,R)$ is in general not generated by the trivial module when $G$ is a $p$-group. Indeed, let $G=C_2$ and $R = \bbZ_2$, then the augmentation ideal of $\bbZ_2[C_2]$ is isomorphic to the sign representation of $C_2$ on $\bbZ_2$, which is not isomorphic to the trivial representation. 
\end{Warn}

\subsection{The stable module category I: the topological model}

We now introduce a version of the stable module category of $R$-linear $G$-representations. This definition is inspired by the description of the stable module category in the case $R=k$ is a field due to Rickard~\cite{Rickard89} and Buchweitz:
\[
\stmod(k[G]) \simeq \cat D^b(k[G])/\Perf(k[G])
\]
as tensor-triangulated categories.

As observed above, for $R$ a general commutative ring, any compact object in $\Fun(BG,\cat D(R))$ is dualizable, so there is an inclusion $\iota\colon \Fun(BG,\cat D(R))^c \subseteq \Fun(BG,\cat D(R))^{\dual}$ between full subcategories of $\Fun(BG,\cat D(R))$.

\begin{Def}\label{def:stmod}
The \emph{stable module category} of $R$-linear $G$-representations is defined as the Verdier quotient
\begin{equation}\label{eq:rickardquotient}
    \stmod(G,R) = \frac{\Fun(BG,\cat D(R))^{\dual}}{\Fun(BG,\cat D(R))^c},
\end{equation}
where we remind the reader that the superscript `$\dual$' indicates the full subcategory of dualizable objects. As such, it inherits a symmetric monoidal structure, which we will also denote by $\otimes = \otimes_R$. We define the big stable module of $R$-linear $G$-representations as the ind-completion of $\stmod(G,R)$, that is
\[
\StMod(G,R) = \Ind\stmod(G,R).
\]
The symmetric monoidal structure on $\stmod(G,R)$ extends to $\StMod(G,R)$. We denote the corresponding finite localization by
\begin{equation}\label{eq:stmodquot}
    \xymatrix{\rho^* = \rho^*(G,R)\colon \Rep(G,R) \ar[r] & \StMod(G,R)}
\end{equation}
and write $\rho_*$ and $\rho^!$ for the corresponding inclusion and its right adjoint, respectively. In summary, we obtain a triple of adjoints $(\rho^*,\rho_*,\rho^!)$.
\end{Def}

\begin{Rem}\label{rem:stmodcomparison}
Other models for the integral stable module category have been given. For example, based on H.~Krause's stable derived category \cite{krause_stablederived}, Balmer and Gallauer \cite{balmergallauer_permutationmodulescohomsing} have studied a version that agrees with the homotopy category of $\StMod(G,R)$. Our choice follows previous work of Mathew \cite{treumannmathew_reps,mathew_torus} and A.~Krause \cite{krause_picard}. Indeed, the embedding $\iota$ considered above exhibits $\Perf(R[G])$ as the thick ideal generated by $R[G]$, so we obtain a canonical equivalence of tt-categories
\[
\stmod(G,R) \simeq \rep(G,R)/\Thickt{R[G]}.
\]
This recovers the definition of the stable module category given in \cite{krause_picard}, where even ring spectrum coefficients are permitted.
\end{Rem}

\begin{Exa}
If $R=k$ is a field, then $\StMod(G,k)$ coincides with the usual stable module category $\StMod(k[G])$ of $k[G]$.
\end{Exa}

Essentially by construction, we obtain a recollement relating the derived category of $R[G]$-modules, the category $\Rep(G,R)$, and the stable module category of $R$-linear $G$-representations.

\begin{Prop}\label{prop:mainrecollement}
The quotient functor $\Rep(G,R) \to \StMod(G,R)$ fits into a recollement
\[
\xymatrix{\StMod(G,R) \ar[r] & \Rep(G,R) \ar[r] \ar@<1ex>[l] \ar@<-1ex>[l] & \cat D(R[G]) \ar@<1ex>[l] \ar@<-1ex>[l]}
\]
of compactly generated stable $\infty$-categories.
\end{Prop}
\begin{proof}
As for example proven in \cite[Prop.~2.2]{bhv2}, this follows by applying ind-categories to the canonical inclusion $\Fun(BG,\cat D(R))^c \to \Fun(BG,\cat D(R))^{\dual}$.
\end{proof}

\begin{Rem}\label{rem:stmodnatural}
By construction, the category $\StMod(G,R)$ is a rigidly-compactly generated tt-category. Formally, the full subcategory of compact objects of the big stable module category $\StMod(G,R)$ is the idempotent completion of $\stmod(G,R)$. In contrast to the case of field coefficients, in general this completion can be non-trivial, i.e., $\stmod(G,R)$ can contain non-split idempotents (\cite[Rem.~4.3]{krause_picard}).
\end{Rem}

The stable module category of \cref{def:stmod} has also been studied in \cite{krause_picard}, see especially Sections 4 and 5. We highlight two observations from this paper which generalize characteristic properties of the stable module categories with field coefficients. To this end, we begin with a brief review of the Tate construction. 

Let $G$ be a finite group. Equipping a spectrum with the trivial $G$-action gives a functor
\[
\triv\colon \Sp \to \Fun(BG,\Sp).
\]
The functor preserves both limits and colimits and consequently admits a left and a right adjoint. These are given by homotopy orbits $(-)_{hG}$ and homotopy fixed points $(-)^{hG}$, respectively. Since $\Sp$ is stable, there is a natural norm transformation $\Nm\colon (-)_{hG} \to (-)^{hG}$. By definition, the cofiber is the $G$-Tate construction:
\[
\xymatrix{X_{hG} \ar[r]^-{\Nm} & X^{hG} \ar[r]^-q & X^{tG}}
\]
for any $X \in \Fun(BG,\Sp)$. The Tate construction $(-)^{tG}\colon \Fun(BG,\Sp) \to \Sp$ is lax symmetric monoidal, so in particular preserves algebra and module structures. Moreover, the natural map $q$ is lax symmetric monoidal (see \cite[\S I.3]{NikolausScholze18}). As expected, there are natural isomorphisms
\begin{equation}\label{eq:tatecohom}
\pi_*R^{hG} \cong H^*(G;R) \quad \text{and } \quad \pi_*R^{tG} \cong \tH^*(G;R)
\end{equation}
for any commutative ring $R$, viewed as a local system via the Eilenberg--MacLane spectrum functor. The next result (see \cite[Lem.~4.2]{krause_picard}) generalizes the fact that the mapping spaces in $\stmod(k[G])$ compute Tate cohomology if $k$ is a field.

\begin{Prop}[Krause]\label{prop:stmodtate}
For any two objects $M,N \in \stmod(G,R)$, there is a natural equivalence
    \[
    \Hom_{\stmod(G,R)}(M,N) \simeq \Hom_{\cat D(R)}(M,N)^{tG}.
    \]
Here, the right hand side denotes the $G$-Tate construction on the derived mapping spectrum $\Hom_{\cat D(R)}(M,N)$ with respect to the induced $G$-action. In particular, we have an equivalence $\End_{\stmod(G,R)}(R) \simeq R^{tG}$ as commutative ring spectra.
\end{Prop}

This implies an integral version of Maschke's theorem:

\begin{Cor}\label{cor:stmodcharacteristic}
If the order of $G$ is invertible in $R$, then $\StMod(G,R)$ is trivial. 
\end{Cor}
\begin{proof}
Let $n$ be the order of $G$ and write $\unit = R$ for the unit in $\StMod(G,R)$. It follows from \cref{prop:stmodtate} and the definition of Tate cohomology that $\pi_0\unit \cong R/n$. If $n$ is invertible in $R$, then $\unit \simeq 0$, so $\StMod(G,R) = 0$.
\end{proof}

We conclude this subsection with a consequence of \cref{thm:eqmoritarep}:

\begin{Cor}\label{cor:eqmoritastmod}
If $R$ is regular, $\StMod(G,R)$ is compactly generated by the (images of the) permutation modules $R[G/H]$ for $H$ ranging through the subgroups of~$G$. Moreover, there is a symmetric monoidal equivalence
\[
\Mod_{\Sp_G}(\borel{R} \otimes \widetilde{EG}) \simeq \StMod(G,R),
\]
where $\widetilde{EG}$ is the cofiber of the canonical map $EG_+ \to S^0$. 
\end{Cor}

For general $R$, we may thus consider $\Mod_{\Sp_G}(\borel{R} \otimes \widetilde{EG})$ as a suitable homotopical variant of the stable module category. Note that the endomorphism of the unit object in both categories are given by the Tate construction $R^{tG}$.

\subsection{The stable module category II: the algebraic model}

The next result concerns the comparison of the stable module with the additive category $\projmod(G,R)$ of $R[G]$-lattices, i.e., $R[G]$-modules whose underlying $R$-module is projective. Equip this category with the exact structure obtained by pulling back the split exact structure on the category $\Proj(R)$ of projective $R$-modules:

\begin{Def}\label{def:exactstructure}
A sequence in $\projmod(G,R)$ is declared to be \emph{exact} if its restriction to $\Proj(R)$ is split exact. 
\end{Def}

In other words, an exact sequence in $\projmod(G,R)$ is exact in the sense of \cref{def:exactstructure} if it is exact in $\Mod(R[G])$ with respect to the $R$-split exact structure.

\begin{Rec}
An $R[G]$-module $M$ is said to be \emph{weakly projective} (resp.~\emph{weakly injective}) if the functor $\Hom_{RG}(M,-)$ (resp.~$\Hom_{RG}(-,M)$) preserves exactness of $R$-split exact sequences. 
\end{Rec}

In light of \cite[Sec.~2]{BensonIyengarKrause13}, the next result is thus not surprising:

\begin{Lem}\label{lem:exactstructure}
Equipped with the exact sequences of \cref{def:exactstructure}, the category $\projmod(G,R)$ forms an exact category. Moreover, the following conditions on an object $M \in \projmod(G,R)$ are equivalent:
    \begin{enumerate}
        \item $M$ is projective with the respect to this exact structure on $\projmod(G,R)$.
        \item $M$ is weakly projective.
        \item $M$ is weakly injective.
        \item $M$ is injective with the respect to this exact structure on $\projmod(G,R)$.
    \end{enumerate}
In particular, $\projmod(G,R)$ is a Frobenius category.
\end{Lem}
\begin{proof}
The verification of the axioms for an exact structure is similar to the one for the $R$-split exact structure on $\Mod(R[G])$. It remains to show the equivalence of Conditions $(a)$--$(d)$. We first prove $(a) \iff (b)$. The implication $(b) \implies (a)$ holds because the exact sequences in $\projmod(G,R)$ are in particular $R$-exact. Conversely, according to \cite[Thm.~2.6]{BensonIyengarKrause13}, a module $M$ is weakly projective if and only if the natural map $M\hspace{-1ex}\uparrow^G \to M$ is split surjective. If $M$ is projective with respect to the structure of \cref{def:exactstructure}, then the identity map on $M$ lifts to the desired splitting, thus $M$ is weakly projective. The proof of $(c) \iff (d)$ is dual, while \cite[Thm.~2.6]{BensonIyengarKrause13} gives $(b) \iff (c)$, thus finishing the proof. 
\end{proof}

For two $R[G]$-modules $M$ and $N$, we say that two morphisms $f,g\colon M \to N$ are homotopic if their difference factors through a coproduct of copies of $R[G]$. The corresponding equivalence relation will be denoted by $\sim$. In light of \cref{lem:exactstructure}, the resulting quotient category is equivalent to the homotopy category of the exact Quillen model structure on $\projmod(G,R)$ associated to the exact structure of \cref{def:exactstructure}:
\begin{equation}\label{eq:frobenius}
    \xymatrix{\projmod(G,R)/\sim \ar[r]^-{\simeq} & \Ho(\projmod(G,R)).}
\end{equation}
Write $\projmod(G,R)^{\fp}$ for the full subcategory of $\projmod(G,R)$ on those $R[G]$-modules which are finitely presented. Both the symmetric monoidal structure and the exact structure restrict, as does the homotopy relation $\sim$.

With this preparation, the following proposition is now essentially a consequence of \cite[Section 4]{krause_picard}. A similar argument is given for field coefficients in \cite{mathew_torus}, in which case the result recovers the classical definition of the stable module category.

\begin{Prop}\label{prop:algebraicmodel}
The composite functor
\begin{equation}\label{eq:smallalgmodelfunctor}
\projmod(G,R)^{\fp} \simeq \Fun(BG,\Proj(R)) \hookrightarrow \Fun(BG,\Perf_R) \twoheadrightarrow \stmod(G,R)
\end{equation}
exhibits the codomain as the quotient of the domain by the homotopy relation $\sim$. Considering the $\infty$-categorical localization, we obtain a symmetric monoidal equivalence of $\infty$-categories
\begin{equation}\label{eq:algmodelfunctor}
    \xymatrix{\big(\projmod(G,R)/\sim\big) \ar[r]^-{\sim} & \StMod(G,R).}
\end{equation}
\end{Prop}
\begin{proof}
Due to \cite[Lem.~4.7]{krause_picard}, it suffices to show that the functor \eqref{eq:smallalgmodelfunctor} is fully faithful. This follows from \cref{lem:exactstructure} combined with \cite[Lem.~4.9]{krause_picard}. The statement about big categories then follows by passing to ind-categories via \cref{rec:indcat}, as in the proof of \cite[Thm.~2.4]{mathew_torus}.
\end{proof}

\begin{Rem}
A version of the stable module category for a finite group with coefficients in a commutative ring has also been constructed and studied by Benson, Iyengar, and Krause in \cite{BensonIyengarKrause13}, via more homological methods. It would be great to clarify the relation between the different versions of the stable module category. 
\end{Rem}

\subsection{The Balmer spectrum of integral representations}\label{ssec:lau}

Via \cref{rem:repmodels}, the Balmer spectrum of $\rep(G,R)$ for regular $R$ has recently been computed by Lau \cite{lau_spcdmstacks}. His theorem forms one of the key ingredients in our approach to the stratification of the integral stable module category.

If $R$ is Noetherian, then so is the graded commutative cohomology ring $H^*(G;R)$, see \cite{evens,golod,venkov}. We write $\Spec^h(H^*(G;R))$ for the homogeneous Zariski spectrum of $H^*(G;R)$. The cohomology ring of any object $M \in \rep(G,R)$ has the structure of a module over $H^*(G;R)$, and we define its \emph{cohomological support}\footnote{This terminology follows Lau~\cite{lau_spcdmstacks}, which differs from the one used in \cite{BensonIyengarKrause08}; in the latter source, the authors refer to this notion of support as the triangulated support.} as
\[
\csupp(M) = \SET{\frak p \in \Spec^h(H^*(G;R))}{\End(M)^*_{\frak p} \neq 0}.
\]

\begin{Thm}[Lau]\label{thm:lau}
Let $R$ be a regular Noetherian ring and $G$ a finite group, then the comparison map  
\begin{equation}\label{eq:laucomparisonmap}
\xymatrix{\Spc(\rep(G,R)) \ar[r]^-{\simeq} & \Spec^h(H^*(G;R))}
\end{equation}
is a homeomorphism. Under this homeomorphism, the Balmer support of any object $M \in \rep(G,R)$ gets identified with its cohomological support.
\end{Thm}

Restriction to the unit element $e \in G$ provides a canonical map $H^*(G;R) \to H^*(e;R) \cong R$, which in turn gives rise to an inclusion $\Spec(R) \hookrightarrow \Spec^h(H^*(G;R))$. We will denote the complement of the image by $\Spec^h(H^*(G;R))\setminus\Spec(R)$. Note that, if $R = k$ is a field, then this complement identifies with the projective variety $\Proj(H^*(G;k))$, as considered in \cite{BensonIyengarKrause11a}.

\begin{Cor}\label{cor:spcstmod}
Let $R$ be a regular Noetherian ring and $G$ a finite group. There is a natural homeomorphism
\[
\Spc(\stmod(G,R)) \cong \left(\Spec^h(H^*(G;R))\setminus\Spec(R)\right)
\]
compatible with the comparison map \eqref{eq:laucomparisonmap}. In particular, there is a bijection
\[
\Supp\colon
\begin{Bmatrix}
\text{Thick tensor ideals} \\
\text{of } \cat \stmod(G, R)
\end{Bmatrix} 
\xymatrix@C=2pc{ \ar[r]^-{\sim} &}
\begin{Bmatrix}
\text{Specialization closed subsets}  \\
\text{of }\Spec^h(H^*(G; R))\setminus\Spec(R)
\end{Bmatrix}.
\]
\end{Cor}
\begin{proof}
By construction (see \cref{def:stmod}), $\stmod(G,R)$ is a finite localization of $\rep(G,R)$, away from the thick tensor ideal generated by $R[G]$. Therefore, viewed as a subspace of $\Spc(\rep(G,R))$, the Balmer spectrum $\Spc(\stmod(G,R))$ is the complement of $\supp(R[G])$. The comparison map \eqref{eq:laucomparisonmap} identifies the support of $R[G]$ with its cohomological support. But $\End_{\rep(G,R)}^*(R[G])$ is isomorphic to $R[G]$ concentrated in degree $0$. Since $R[G]$ is finite and free over $R$, we compute
\[
\csupp(R[G]) = \SET{\frak p \in \Spec^h(H^*(G;R))}{H^{>0}(G;R) \subseteq \frak p}. 
\]
It follows that $\supp(R[G])$ identifies with $\Spec(R)$, as desired.
\end{proof}

\begin{Exa}\label{ex:bcr}
If $R=k$ is a field of characteristic $p$, then \cref{cor:spcstmod} recovers a theorem of Benson, Carlson, and Rickard \cite{BensonCarlsonRickard97}: The support variety of $G$ over $k$ is given by 
\[
\Spc(\stmod(G,k)) \cong \Proj(H^*(G;k)).
\]
Indeed, unwinding the definitions, the image of $\Spec(k) \to \Spec^h(H^*(G;k))$ corresponds to the `irrelevant' ideal $H^{>0}(G;k) \subset H^{*}(G;k)$.
\end{Exa}

By Lau's theorem and \cref{cor:spcstmod}, both $\rep(G,R)$ and $\stmod(G,R)$ thus have Noetherian Balmer spectrum and hence satisfy the local-to-global principle by \cite[Thm.~6.9]{Stevenson13} or \cite[Thm.~3.21]{bhs1}.

\section{\'Etale descent for stable module categories}

The goal of this section is to prove that in order to establish stratification for $\Rep(G,R)$ and $\StMod(G,R)$ it suffices to consider the case of elementary abelian groups $E = (\bbZ/p)^{\times r}$ for all primes $p$ and ranks $r$. To this end, we use a result of Balmer which allows us to recognize restriction functors as \'etale extensions and then apply the finite \'etale descent theorem (\cref{thm:etaledescent}). A key input will be a derived version of Chouinard's theorem, which we deduce from work of Carlson.

\begin{Conv}
Throughout this section, $R$ and $S$ denote commutative rings, later assumed to be Noetherian. We typically do not impose any regularity conditions. 
\end{Conv}

\subsection{Base-change functors}\label{ssec:basechange}

We start by setting up the basic base-change functors for (derived) categories of representations. Throughout, we will assume some familiarity with abstract properties and compatibilities afforded by adjunctions arising from tt-functors, as for example systematically developed in \cite{BalmerDellAmbrogioSanders15}.

\subsubsection{Change of coefficients}

Let $f\colon R \to S$ be a homomorphism of commutative rings. There are induced induction and restriction functors
\[
f^*\colon \Mod(R[G]) \to \Mod(S[G]), \quad f_*\colon \Mod(S[G]) \to \Mod(R[G])
\]
given by $S\otimes_R-$ and forgetting along $f$, respectively. As usual, $f^*$ is symmetric monoidal and preserves compact objects, so the adjunction $(f^*,f_*)$ satisfies the projection formula. If $S$ is projective as an $R$-module, this adjunction restricts to an adjunction between $\projmod(G,R)$ and $\projmod(G,S)$. 

Induction is exact and derives to the base-change functor $f^*\colon \Fun(BG,\cat D(R)) \to \Fun(BG,\cat D(S))$ given by post-composition. Here and henceforth, we use the same symbols for exacts functors and their corresponding derived versions. Because $f^*$ is symmetric monoidal, it restricts to a functor on dualizable objects, $f^*\colon \rep(G,R) \to \rep(G,S)$ and then extends to a symmetric monoidal functor between ind-categories. By construction, the latter functor is cocontinuous and preserves compact objects, so it admits a right adjoint $f_*$, which has another right adjoint $f^!$. The resulting triple of adjoints
\[
\xymatrix{f^*,f^!\colon \Rep(G,R) \ar@<1ex>[r] \ar@<-1ex>[r] & \Rep(G,S)\noloc f_* \ar[l]}
\]
in particular satisfies the projection formula. We will occasionally simplify notation and omit the forgetful functor $f_*$ from the notation.

\begin{Rem}\label{rem:etalecoeffbasechange}
Base-change along a map of commutative rings $f\colon R \to S$ induces a symmetric monoidal equivalence
\[
\Mod_{\Rep(G,R)}(S) \simeq \Rep(G,S),
\]
where $\Mod_{\Rep(G,R)}(S)$ is the category of modules over $f_*S$ internal to $\Rep(G,R)$. Under this equivalence, the functors $f^*$ and $f_*$ identify with induction and restriction, respectively. Moreover, localization yields the analogous statement on stable module categories. 
\end{Rem}

\subsubsection{Change of group}\label{ssec:changeofgroups}

If $H$ is a subgroup of $G$, then there is a restriction functor 
\[
\res_H\colon \Mod(R[G]) \to \Mod(R[H]).
\]
Restriction admits a left and a right adjoint, given by induction $\ind_H$ and coinduction $\coind_H$, respectively. As for field coefficients, induction and coinduction are naturally equivalent; we will occasionally use this isomorphism implicitly to identify these functors. Since $R[G]$ is a free $R[H]$-module, induction, restriction, and coinduction naturally restrict to representations whose underlying $R$-module is projective. We will denote the corresponding functors by the same symbols.

Both restriction and (co)induction are exact functors and thus give rise to a corresponding adjunction of functors between derived categories:
\[
\xymatrix{\res_H\colon \rep(G,R) \ar@<0.5ex>[r] & \rep(H,R)\noloc \coind_H. \ar@<0.5ex>[l]}
\]
For ease of notation, we will occasionally denote the inclusion $i\colon H \subseteq G$ and then write $\res_H = i^*$ and $\coind_H = i_*$. This adjunction extends to an adjunction upon passage to ind-categories, which we again denote by the same symbols:
\[
\xymatrix{\res_H\colon \Rep(G,R) \ar@<0.5ex>[r] & \Rep(H,R)\noloc \coind_H \ar@<0.5ex>[l]}
\]
Because $\res_H$ is symmetric monoidal, this adjunction fits into an infinity sequence of adjunctions as in \cite{BalmerDellAmbrogioSanders16}; in particular, there is a projection formula for objects $X \in \Rep(G,R)$ and $Y \in \Rep(H,R)$:
\begin{equation}\label{eq:projformula}
    i_*(i^*(X) \otimes Y) \simeq X \otimes i_*(Y)
\end{equation}
and both $i^*$ and $i_*$ preserve all colimits and limits. The next result is a minor variation on a result due to Balmer.

\begin{Prop}[Balmer]\label{prop:restrictionetale}
For any subgroup $i\colon H \subseteq G$, the functor $i^*$ is a finite \'etale extension of finite degree and there is a symmetric monoidal equivalence
\[
\Rep(H,R) \simeq \Mod_{\Rep(G,R)}(i_*R).
\]
Under this equivalence, the adjunction $(i^*,i_*)$ identifies with the (induction, restriction) adjunction along the map of commutative algebra objects $R_G \to i_*R_H\simeq i_*i^*R_G$ in $\Rep(G,R)$.
\end{Prop}
\begin{proof}
In \cite[Thm.~4.3]{Balmer15}, Balmer proves the analogue of this result for the derived tt-category $\Fun(BG,\cat D(R))$, with the exception of the finite degree claim. His proof works equally well for $\Rep(G,R)$, see also his Remark 4.5; we omit the details.

To verify that the resulting algebra object $i_*i^*R_G \in \rep(G,R)$ has finite degree, we can proceed as in \cite{Balmer14}: By his Theorem 3.7(b), it suffices to produce a conservative tt-functor $F\colon\rep(G,R) \to \cat L$ such that $F(i_*i^*R_G)$ has finite degree. We claim that the forgetful functor $\rep(G,R) \to \Perf(R)$ has these properties. Indeed, it is conservative by \cref{lem:deraugconservative}, and any tt-ring in $\Perf(R)$ has finite degree by \cite[Cor.~4.3]{Balmer14}. Therefore, our claim is true for $\rep(G,R)$ by virtue of \cite[Thm.~3.7(b)]{Balmer14}.
\end{proof}

Moreover, we can see that the $(i^*,i_*)$ adjunction is compatible with base-change, in the following sense:

\begin{Prop}\label{prop:basechangecompatibilities}
Let $f\colon R \to S$ be a map of commutative rings and $i\colon H \subseteq G$ a subgroup inclusion. Then there are naturally commutative diagrams
\[
\xymatrix{\Rep(G,R) \ar[r]^-{i^*} \ar[d]_{f^*} & \Rep(H,R) \ar[d]^{f^*} & \Rep(G,R) \ar[d]_{f^*} & \Rep(H,R) \ar[d]^{f^*} \ar[l]_-{i_*} \\
\Rep(G,S) \ar[r]_-{i^*} & \Rep(H,S) & \Rep(G,S) & \Rep(H,S) \ar[l]^-{i_*}}
\]
which localize to the corresponding commutative diagrams for the stable module category. Moreover, the finite localizations of \eqref{eq:stmodquot} fit into  commutative squares
\begin{equation}\label{eq:rhobasechange}
    \vcenter{
    \xymatrix{\StMod(G,R) \ar[r]^-{i^*} & \StMod(H,R) & \StMod(G,R) \ar[r]^-{f^*} & \StMod(G,S) \\
    \Rep(G,R) \ar[r]_-{i^*} \ar[u]^{\rho^*} & \Rep(H,R) \ar[u]_{\rho^*} & \Rep(G,R) \ar[r]_-{f^*} \ar[u]^{\rho^*} & \Rep(H,S). \ar[u]_{\rho^*}}
    }
\end{equation}
\end{Prop}
\begin{proof}
The commutativity of the top left square expresses the fact that precomposition commutes with postcomposition in functor categories. To see that the top right square commutes as well, one can either directly verify the dual Beck--Chevellay condition to establish the right adjointability of the left square. Alternatively, observe that both horizontal functor are based changed from the case of $\bbZ$-coefficients, hence compatible with $f^*$. 

In order to see that these two squares localize, observe that we have $i_*R[H] \simeq R[G]$ in $\Rep(G,R)$, while $i^*(R[G])$ decomposes as a finite direct sum of $R[H]$ as an object of $\Rep(H,R)$, and likewise for coefficients in $S$. Finally, the thick tensor ideals generated by $R[G]$ and $R[H]$ correspond under $i^*$, hence induce a tt-functor $i^*\colon \StMod(R,G) \to \StMod(R,H)$ making the left square in \eqref{eq:rhobasechange} commute. Likewise, $f^*R[G] \cong S[G]$, so the right square of \eqref{eq:rhobasechange} commutes as well.
\end{proof}

\subsection{Deriving Chouinard}\label{ssec:chouinard}

Let $R$ be any ring and $G$ a finite group. Chouinard's theorem \cite[Cor.~1.1]{Chouinard76} says that an $R[G]$-module is projective if and only if all of its restrictions to elementary abelian subgroups of $G$ are projective. Interpreted as a statement about the stable module category and restricting to $R = k$ a field, it is equivalent to the joint conservativity of the functors $\res\colon \StMod(k[G]) \to \StMod(k[E])$ ranging over all elementary abelian subgroups $E$ of $G$. The goal of this section is to prove a version of Chouinard's theorem for $\Rep(G,R)$. 

\begin{Not}
We write $\cat E(G)$ (resp.~$\cat E_p(G)$) for the collection of elementary abelian subgroups (resp.~elementary abelian $p$-subgroups) of $G$.
\end{Not}

The key input to the proof of our version of Chouinard's theorem for the categories $\Rep$ is a result of Carlson \cite{Carlson00}, which we restate here for the convenience of the reader.

\begin{Thm}[Carlson]\label{thm:carlson}
There exists $\tau = \tau(G) \in \bbN$ and a finitely generated $\bbZ[G]$-module $V$ together with a filtration
\[
(0) = L_0 \subseteq L_1 \subseteq \ldots \subseteq L_{\tau} = \bbZ \oplus V
\]
such that for each $j = 1,\ldots,\tau$ there is an $E_j \in \cat E(G)$ and $W_j \in \projmod(E_j,\bbZ)$ with $L_j/L_{j-1} \cong \ind_{E_j}W_j$ as $\bbZ[G]$-modules.
\end{Thm}

For each elementary abelian subgroup $E$ of $G$ we have a tt-functor given by restriction. Varying over $\cat E(G)$, these functors assemble into a tt-functor $\res$, which will later be used to descend stratification along. One of the key ingredients is the next result, which is a version of Chouinard's theorem appropriate for $\Rep(G,R)$.

\begin{Thm}\label{thm:derivedchouinard}
Let $R$ be a commutative ring. The restriction functor
\[
\xymatrix{\res\colon \Rep(G,R) \ar[r] & \prod_{E \in \cat E(G)}\Rep(E,R)}
\]
is finite \'etale, of finite degree, and conservative.
\end{Thm}
\begin{proof}
The individual components of the restriction functors were shown to be finite \'etale extensions of finite degree in \cref{prop:restrictionetale}. Since $\cat E(G)$ is finite, $\res$ is also finite \'etale of finite degree, and it remains to prove conservativity. 

To this end, note that Carlson's theorem \cref{thm:carlson} implies that there exist modules $W_E \in \rep(G,\bbZ)$ for each $E \in \cat E(G)$ such that
\[
\bbZ \in \Thick(\SET{i_*W_E}{i\colon E \subseteq G, E \in \cat E(G)}).
\]
Thanks to the compatibility of coinduction with base-change (\cref{prop:basechangecompatibilities}), we obtain modules $V_E = R\otimes_{\bbZ}W_E \in \rep(G,R)$ with 
\begin{equation}\label{eq:gencarlson}
    R \in \Thick(\SET{i_*V_E}{i\colon E \subseteq G, E \in \cat E(G)}).
\end{equation}
Now consider $X \in \Rep(G,R)$ with $\res(X) = 0$, so $i^!X \simeq 0$ for each elementary abelian subgroup $i\colon E \subseteq G$. It follows from the projection formula~\eqref{eq:projformula} that 
\[
0 \simeq X \otimes i_*V_E \simeq i_*(i^*(X)\otimes V_E)
\]
for all $E \in \cat E(G)$. Therefore, $X \otimes V \simeq 0$ for every $V \in \Thick(\SET{i_*V_E}{i\colon E \subseteq G, E \in \cat E(G)})$, hence $X \simeq X \otimes R \simeq 0$ by \cref{eq:gencarlson}. This concludes the proof.
\end{proof}

\begin{Cor}\label{cor:nilfaithful}
The functor $\res$ is nil-faithful. 
\end{Cor}
\begin{proof}
This follows from \cref{thm:derivedchouinard} and \cite[Prop.~3.15]{Balmer16b}. 
\end{proof}

\begin{Not}
With notation as in \cref{thm:derivedchouinard}, we write 
\begin{equation}\label{eq:descendableA}
    \cA(G,R) = \prod_{E \in \cat E(G)}i_*i^*R.    
\end{equation} 
Since $i_*$ is lax monoidal as the right adjoint of the symmetric monoidal functor $i^*$, the object $\cA(G,R)$ canonically has the structure of a commutative algebra in $\Rep(G,R)$.
\end{Not}

\begin{Cor}\label{cor:descendable}
The commutative algebra $\cA(G,R) \in \rep(G,R)$ is descendable.
\end{Cor}

\begin{Rem}
This result generalizes Theorem 4.3 in Balmer's \cite{Balmer16b}, which there was proven over a field of characteristic $p$; both proofs rely on Serre's vanishing of Bockstein's theorem. It is also representation-theoretic counterpart of \cite[Prop.~5.25]{mnn_derivedinduction}. For $R$ regular, both statements are equivalent by virtue of \cref{thm:eqmoritarep}.
\end{Rem}

\subsection{Generalized Quillen stratification}

We record two consequences of \cref{thm:derivedchouinard}, one concerning a generalization of Quillen stratification for group cohomology and the other one providing the desired \'etale descent theorem for stratification.

\begin{Thm}\label{thm:genquillenstrat}
Let $R$ be a commutative ring and write $\orbit_{\cat E}(G)$ for the $G$-orbit category on the elementary abelian subgroups of $G$. The restriction functors then induce a homeomorphism
\[
\xymatrix{\overline{\varphi}\colon \colim_{\orbit_{\cat E}(G)}\Spc(\rep(E,R)) \ar[r]^-{\simeq} & \Spc(\rep(G,R)).}
\]
\end{Thm}

This result is an integral generalization of \cite[Thm.~4.10]{Balmer16b}. As explained there, one may use it to transfer generalized Quillen stratification to any tt-category $\cat K$ that receives a tt-functor from $\rep(G,R)$, such as the stable module category for instance. In lieu of a generalization of Lau's theorem to non-regular coefficient rings, it might also give a new approach to the computation of $\Spc(\rep(G,R))$. Another consequence of \cref{thm:genquillenstrat} is the following version of Quillen stratification for $H^*(G;R)$ for regular $R$:

\begin{Cor}\label{cor:quillenstratification}
If $R$ is a regular commutative ring and $G$ a finite group, then there is a homeomorphism
\[
\xymatrix{\colim_{\orbit_{\cat E}(G)}\Spec^h(H^*(E;R)) \ar[r]^-{\simeq} & \Spec^h(H^*(G;R)).}
\]
\end{Cor}
\begin{proof}
This follows from \cref{thm:genquillenstrat} and Lau's theorem \cref{thm:lau}, because of the naturality of Balmer's comparison map, see \cite[Thm.~5.3]{Balmer10b}. 
\end{proof}

The next result establishes finite \'etale descent for stratification from finite groups to elementary abelian groups. 

\begin{Thm}\label{thm:descenttoelab}
Let $R$ be a Noetherian commutative ring and $G$ a finite group. 
\begin{enumerate}
    \item If $\Rep(E,R)$ is stratified for all $E \in \cat E(G)$, then so is $\Rep(G,R)$.
    \item If $\StMod(E,R)$ is stratified for all $E \in \cat E(G)$, then so is $\StMod(G,R)$.
\end{enumerate}
\end{Thm}
\begin{proof}
\cref{thm:derivedchouinard} shows that $\res$ satisfies the assumptions of \cref{thm:etaledescent}, hence $\Rep(G,R)$ is stratified if the tt-categories $\Rep(E,R)$ are stratified for all $E \in \cat E(G)$.

Since $\Spc(i^*)^{-1}(\supp(R[G])) = \supp(i^*R[G]) = \supp(R[H])$, the left square in~\eqref{eq:rhobasechange} satisfies the condition of the base-change criterion for finite \'etale tt-functors established in \cite[Ex.~5.12]{Sanders21pp}. It follows that the functor $i^*$ between stable module categories is also finite \'etale. Moreover, the corresponding descendable algebra object is $\rho^*(\cA(G,R)) \in \CAlg(\stmod(G,R))$. \Cref{thm:etaledescent} finishes the proof of Part (b).
\end{proof}

\begin{Rem}
This theorem is analogous to the reduction step of \cite{BensonIyengarKrause11a}. However, we note that, in contrast to their proof, we have not used any form of Quillen stratification; instead, the arguments are tt-theoretic and rely on \cref{thm:etaledescent}. In fact, we do not know in which generality Quillen's stratification theorem (in the form of \cite[Thm.~10.2]{Quillen71}) holds. It is plausible that at least for Dedekind domains, the assertion can be assembled by hand from the field case.
\end{Rem}

\section{Modular lifting for elementary abelian groups}

In this section we will study stratification for elementary abelian groups. Typically, we will fix a prime $p$, denote elementary $p$-groups by the letter $E$ and write $r = \rank(E)$ for their rank. In addition, the following convention will be in place:

\begin{Conv}
Throughout this section, we will take the coefficient ring $A$ to be a complete discrete valuation ring (DVR) of mixed characteristic and with perfect residue field $k$ of characteristic $p$. 
\end{Conv}

The structure of such rings is well understood. According to \cite[Thm.~II.4]{serre_localfields}, there exists a unique ring homomorphism $j\colon W(k) \to A$ which commutes with the reduction modulo the maximal ideal. Here, $W(k)$ denotes the ring of $p$-typical Witt vectors on $k$, where $p$ is the characteristic of $k$. Via $j$, the $W(k)$-module $A$ is free of rank the ramification index $e$ of $A$. With $K$ denoting the fraction field of $A$, the triple $(K,A,k)$ is also commonly referred to as a $p$-modular system in the representation theory literature.

\subsection{Cohomology of elementary abelian groups}

We collect some basic facts about the cohomology of elementary abelian groups. Let us fix a prime $p$. 

\begin{Lem}\label{lem:elabcohom}
Let $E$ be an elementary abelian $p$-group. The integral cohomology satisfies the following properties:
    \begin{enumerate}
        \item $H^0(E;A) \cong A$;
        \item $H^1(E;A) = 0$;
        \item $p\cdot H^{*>0}(E;A) = 0$.
    \end{enumerate}
\end{Lem}
\begin{proof}
The first property is true by definition, while (b) holds since $H^1(E;A) \cong \Hom(E,A) = 0$ as $A$ has characteristic $0$. The last statement is proven by induction on $r=\rank(E)$. The base case $r=1$ is a direct computation, while the induction step follows from the (non-canonically split) K\"unneth short exact sequence. 
\end{proof}

\begin{Rem}
By a theorem of Adem \cite[Thm.~2.1]{adem_exponent}, a finite $p$-group $G$ satisfies $p\cdot H^{*>0}(G;\Z) = 0$ if and only if it is elementary abelian. 
\end{Rem}

\begin{Lem}\label{lem:elabtatecohom}
If $E$ is an elementary abelian $p$-group of rank $r$, then $p^r\cdot\tH^*(E;A)=0$. In particular, multiplication by $p^r$ is null on $A^{tE}$.
\end{Lem}
\begin{proof}
This result follows from the definition of Tate cohomology. Indeed, by \cref{lem:elabcohom}(c), we have $p\cdot\tH^*(E;A) = 0$ for $* \neq -1,0$. Since the action of $E$ on $A$ is trivial, the norm map is multiplication by $p^r$ on $A$, which implies the claim for $*=-1,0$. Now $A^{tE}$ is a commutative ring spectrum with $p^r$ acting by zero on $\pi_0A^{tE} \cong \tH^0(E;A)$, see \eqref{eq:tatecohom}, hence $p^r$ is null on $A^{tE}$.
\end{proof}

\begin{Rem}
The Tate cohomology of an elementary abelian $p$-group is Noetherian if and only if its rank is $r=1$. This follows from the fact that a finite abelian group which is not cyclic will not have periodic cohomology. 
\end{Rem}

\subsection{Reduction to the residue field}

Recall that $A$ is a complete DVR of mixed characteristic with maximal ideal $\frak p$ and residue field $A/\frak p=k$ of characteristic $p$. Let $\pi$ be a uniformizer of $A$ so that $(\pi) = \frak p$ and denote by $e$ the ramification index of $p$ in $A$, so that we can write $pA = \frak p^e$. We continue to let $E$ be an elementary abelian $p$-group. The goal of this subsection is to prove the following result:

\begin{Prop}\label{prop:elabtategeneration}
In the category of $A^{tE}$-modules, we have $A^{tE} \in \Thick(k^{tE})$.
\end{Prop}

We record a direct consequence. Write $\alpha\colon A^{tE} \to k^{tE}$ for the map of commutative ring spectra induced by the projection $A \to k$.

\begin{Cor}\label{cor:elabtateconservative}
Both induction and coinduction $\alpha^*,\alpha^!\colon \Mod(A^{tE})\to\Mod(k^{tE})$ along $\alpha$ are conservative.
\end{Cor}

\begin{proof}[Proof of \cref{prop:elabtategeneration}]
We will prove this result in three steps: First, we reduce to $(A/p^r)^{tE}$ and then to $k^{tE}$. To simplify notation, we will denote the Eilenberg--MacLane spectrum associated to a discrete ring $R$ by $R$ as well and thus work intrinsically to the $\infty$-category of spectra.

Consider the cofiber sequence $A \xrightarrow{p^r} A \to A/p^r$ in $\Mod(A)$. Since $(-)^{tE}$ is lax symmetric monoidal, we obtain a cofiber sequence
\[
\xymatrix{A^{tE} \ar[r]^-{p^r} & A^{tE} \ar[r] & (A/p^r)^{tE}}
\]
in $\Mod(A^{tE})$. According to \cref{lem:elabtatecohom}, multiplication by $p^r$ is null on $A^{tE}$, so the above cofiber sequence exhibits $A^{tE}$ as a retract of $(A/p^r)^{tE}$.

Since $(p) = (\pi)^e$, we get an isomorphism $A/p^r \cong A/\pi^{e \cdot r}$. The $A$-module $A/\pi^{e\cdot r}$ admits a (non-split) finite filtration
\begin{equation}\label{eq:finitefiltration}
0 \subset A/\pi \subset A/\pi^2 \subset \ldots \subset A/\pi^{e\cdot r-1} \subset A/\pi^{e\cdot r}
\end{equation}
with filtration quotients $A/\pi\cong k$. Applying the $E$-Tate construction to the corresponding cofiber sequences gives rise to cofiber sequences in $\Mod(A^{tE})$,
\[
(A/\pi^{j-1})^{tE} \to (A/\pi^{j})^{tE} \to k^{tE},
\]
for $1\le j \le r\cdot e$. These cofiber sequences exhibit $(A/\pi^{r\cdot e})^{tE} \in \Thick(k^{tE})$. In summary, we see that $A^{tE} \in \Thick(k^{tE})$ in the category of $A^{tE}$-modules.
\end{proof}

\subsection{Generic modular lifting}

In this subsection, we will use nil-descent to establish stratification for $\StMod(E,A)$ for $E$ an elementary abelian $p$-group and $A$ a complete DVR of mixed characteristic with perfect residue field $k$ of characteristic~$p$. This result may be interpreted as a form of modular lifting for representations from $k$ to $A$ up to projective representations and tt-operations. 

We first establish an auxiliary result. Let $K$ be the fraction field of $A$. The augmentation $A^{hE} \to A$ induces an injection on homogeneous spectra
\[
\Spec(A) \to \Spec^h(H^*(E;A)).
\]
We denote by $\eta_{K}$ the image of $(0)$ under this map.

\begin{Lem}\label{lem:quotmapelabcohom}
The quotient map $\alpha\colon A \to k$ induces an injective continuous map on homogeneous spectra $\Spec(H^*(E;k)) \to \Spec^h(H^*(E;A))$
whose image misses precisely the point $\eta_{K}$.
\end{Lem}
\begin{proof}
The map $\alpha\colon A^{hE} \to k^{hE}$ factors as a composite
\[
\xymatrix{A^{hE} \ar[r]^-{\alpha_1} & (A/p^r)^{hE} \ar[r]^-{\alpha_2} & k^{hE}.}
\]
First, consider the long exact sequence induced by $\alpha_1$:
\[
\xymatrix{\ldots \ar[r] & H^s(E;A) \ar[r]^-{\cdot p^r} & H^s(E;A) \ar[r]^-{H^s(\alpha_1)} & H^s(E;A/p^r) \ar[r] & \ldots.}
\]
Let $\Lambda(\epsilon)$ be an exterior algebra over $A$ with $\epsilon$ in degree -1. By virtue of \cref{lem:elabcohom}, this gives an isomorphism $H^s(E;A/p^r) \cong \Lambda(\epsilon) \otimes H^s(E;A)$ in positive degrees, while in degree $0$ it identifies with the projection $A \to A/p^r$. It follows that $\alpha_1$ induces an injection on homogeneous spectra that misses precisely the point $\eta_{K}$ corresponding to $K=\pi^{-1}A$ in degree 0. 

Second, using induction on the finite filtration \eqref{eq:finitefiltration}, we see that the ring map $H^*(E;\alpha_2)\colon H^*(E;A/p^r) \to H^*(E;k)$ is an $F$-isomorphism. Therefore, the induced map $\Spec^h(H^*(E;\alpha_2))$ is a homeomorphism, so the claim follows.
\end{proof}

The next result was originally proven by Benson, Iyengar, and Krause \cite[Thm.~8.1]{BensonIyengarKrause11a} using the Bernstein--Gelfand--Gelfand correspondence (see \cite{abim_bgg}). For completenes, we include the sketch of an alternative, homotopy-theoretic argument due to Mathew, which appeared in his unpublished manuscript \cite{mathew_torus}.

\begin{Thm}[Benson--Iyengar--Krause]\label{thm:elabbikstrat}
Let $k$ be a field of characteristic $p$ and $E$ an elementary abelian $p$-group of rank $r$, then the category $\StMod(E,k)$ is stratified over $\Spec^h(H^*(E;k))\setminus \Spec^h(k) \cong \mathbb{P}_k^{r-1}$.
\end{Thm}
\begin{proof}[Sketch of alternative proof (Mathew)]
Since $E$ is a $p$-group, there is a canonical symmetric monoidal equivalence $\StMod(G,k) \simeq \Mod(k^{tE})$. The standard inclusion $E\cong (\bbZ/p)^{\times r} \to (S^1)^{\times r}$ induces a morphism of commutative ring spectra $k^{t(S^1)^{\times r}} \to k^{tE}$. This map exhibits $k^{tE}$ as a faithful Galois extension of $k^{t(S^1)^{\times r}}$ with Galois group $(S^1)^{\times r}$ in the sense of Rognes \cite{Rognes08}.

Since the Galois group is a connected finite complex, by \cite[Thm.~6.5]{mathew_torus} the localizing tensor ideals of $\Mod(k^{tE})$ are in one-to-one correspondence with the localizing tensor ideals of $\Mod(k^{t(S^1)^{\times r}})$. We compute $\pi_*k^{t(S^1)^{\times r}} \cong k[z_1^{\pm 1},\ldots,z_r^{\pm 1}]$ with all $z_i$ of degree~2, so this graded commutative ring is regular Noetherian and concentrated in even degrees. This implies that  $\Mod(k^{t(S^1)^{\times r}})$ is stratified over projective space $\mathbb{P}_k^{r-1}$, so the claim follows as in \cite[Thm.~6.8]{mathew_torus}.
\end{proof}

\begin{Thm}\label{thm:elabstmodstratification}
Let $A$ be a complete DVR of mixed characteristic $(0,p)$ and $E$ an elementary abelian $p$-group. Then the tt-category $\StMod(E,A)$ is stratified.
\end{Thm}
\begin{proof}
Let $\alpha\colon A \to k$ be the quotient map and denote the induced tt-functor by $\alpha^*\colon \StMod(E,A) \to \StMod(E,k)$. Since $\alpha^*$ is cocontinuous, it admits a right adjoint $\alpha_*$, which in turn has a further right adjoint $\alpha_!$, see \cref{ssec:basechange}. Our goal is to apply nil-descent (\cref{thm:nildescent}) to $\alpha^*$. 

In preparation for the verification of the conditions of nil-descent, consider the localizing subcategory $\Loc(A) \subseteq \StMod(E,A)$. By derived Morita theory and \cref{prop:stmodtate}, we have an equivalence 
\[
\Mod(A^{tE}) \simeq \Loc(A),
\]
where the $E$-Tate construction is taken with respect to the trivial $E$-action on $A$. In particular, \cref{prop:elabtategeneration} transposes to the statement that $A \in \Thick(k)$ in $\StMod(E,A)$.

In particular, we see that $\alpha^*$ and $\alpha^!$ are conservative. In order to show that $\alpha_*$ is also conservative, we will use that $\StMod(H,k)$ is generated by the trivial $H$-representation $k$ for any finite $p$-group $H$. This yields a commutative diagram
\[
\xymatrix{\StMod(E,A) & \StMod(E,k) \ar[l]_-{\alpha_*}\\
\Mod(A^{tE}) \ar@{^{(}->}[u] & \Mod(k^{tE}) \ar[u]_{\simeq} \ar[l]^-{\alpha_*}.}
\]
The bottom horizontal map in this diagram is restriction along $\alpha\colon A^{tE}\to k^{tE}$, so it is conservative; hence the top horizontal functor is conservative as well. 

Next we show that the map $\varphi\colon \Spc(\stmod(E,k)) \to \Spc(\stmod(E,A))$ induced by $\alpha^*$ is bijective. To this end, recall that $\alpha^*$ is compatible with the finite localizations $\rho^*$ as in the left base-change square of \eqref{eq:rhobasechange}. The induced diagram of Balmer spectra takes the form:
\[
\xymatrix{\Spc(\stmod(E,A)) \ar@{^{(}->}[d]_{\Spc(\rho^*)} & \Spc(\stmod(E,k)) \ar@{^{(}->}[d]^{\Spc(\rho^*)} \ar[l]_-{\varphi} \\
\Spc(\rep(E,A))  & \Spc(\rep(E,k)). \ar[l]^-{\widetilde{\varphi}}}
\]
By Lau's theorem (\cref{thm:lau}) and naturality of Balmer's comparison map, the map $\widetilde{\varphi}$ identifies with the injection from \cref{lem:quotmapelabcohom}. It follows that $\varphi$ is injective. To see that it also surjective, note that $k$ is compact as an $A$-module, so $\im(\varphi) = \supp(\alpha_*k)$ by \cite[Thm.~1.7]{Balmer18}. Since $A \in \Thick(k)$, we deduce $\Spc(\stmod(E,A)) = \supp(A) \subseteq \supp(\alpha_*k)$, which implies the claim. 

In summary, $\alpha^*$ satisfies the conditions of \cref{thm:nildescent}. We can therefore descend stratification from $\StMod(E,k)$, which was established in \cite{BensonIyengarKrause11a}, to the stable module category $\StMod(E,A)$.
\end{proof}

\begin{Cor}\label{cor:elabgenericlifting}
Let $A$ be a complete DVR of mixed characteristic $(0,p)$ and $E$ an elementary abelian $p$-group. The reduction map $A \to k$ induces a bijection
\begin{equation}\label{eq:lifting}
\begin{Bmatrix}
\text{Localizing tensor ideals} \\
\text{of } \cat \StMod(E,A)
\end{Bmatrix} 
\xymatrix@C=2pc{ \ar[r]^-{\sim} &}
\begin{Bmatrix}
\text{Localizing tensor ideals} \\
\text{of } \cat \StMod(E,k)
\end{Bmatrix}.
\end{equation}
\end{Cor}
\begin{proof}
We concatenate the classifying bijections \eqref{eq:lococlassification} of \cref{thm:elabstmodstratification} and \cref{thm:elabbikstrat} with the bijection of spectra $\varphi\colon\Spc(\stmod(E,k)) \xrightarrow{\simeq} \Spc(\stmod(E,A))$ established in the proof of \cref{thm:elabstmodstratification}. Unwinding the definition, we see that the resulting bijection is induced by reduction to $k$.
\end{proof}

\begin{Rem}
A $k$-linear representation $V$ of a finite group $G$ is said to be \emph{liftable} if there exists $W \in \projmod(G,A)$ such that $W\otimes_Ak \cong V$. In general, it is a subtle problem to decide which representations are liftable. As one instance of a positive result in this direction, all endotrivial modules are liftable as shown in \cite{lms_lifting}, extending earlier work of Alperin \cite{alperin_lifting}. We may interpret the bijection \eqref{eq:lifting} as saying that modular lifting for elementary abelian groups holds \emph{generically}, i.e., up to the structural operations available in a tt-category. 
\end{Rem}

\section{Stratification and its consequences}

In this section, we will prove our main stratification theorems for the stable module category and $\Rep(G,\cat O)$, for any finite group $G$ and Dedekind domain $\cat O$ of characteristic $0$. Besides the results of the previous section, we employ a local-to-global principle for the stable module category, which is the subject of \cref{ssec:arithlgt}.~In \cref{ssec:stratrepiffstmod}, we combine Lau's theorem with nil-descent and Neeman's stratification theorem to show that the stable module category $\StMod(G,R)$ is stratified if and only if $\Rep(G,R)$ is stratified. Consequences for $\projmod(G,\cat O)$ are given in \cref{ssec:stratreps}.

\subsection{An arithmetic local-to-global principle}\label{ssec:arithlgt}

Inspired by the Hasse square in algebraic number theory, we establish an arithmetic local-to-global principle for representations. This technique will allow us to globalize the stratification of the stable module category for complete DVRs to characteristic $0$ Dedekind domains. 

Let $K$ be an algebraic number field, then its ring of integers $\cat O = \cat O_K$ is a Dedekind domain, i.e., an integrally closed, Noetherian domain with Krull dimension~1. In particular, $\cat O$ is regular and every non-zero prime ideal is maximal. The localization of $\cat O$ at a maximal ideal $\frak p$ is a DVR, and we write $\cdvr{\cat O}{\frak p}$ for the corresponding completion at $\frak p$. The latter is then a complete DVR of mixed characteristic $(0,p)$ with finite residue field $k$ of characteristic $p$. We note that any prime number $q$ decomposes in $\cat O$ uniquely as
\begin{equation}\label{eq:primedecomposition}
(q) = \prod_{i=1}^g\frak q_i^{e_i}
\end{equation}
for non-zero prime ideals $\frak q_i$. The natural numbers $e_i$ are known as the ramification index of $\frak q_i$ over $q$.

\begin{Conv}\label{conv:o}
Throughout this section, $\cat O$ will denote a Dedekind domain with fraction field $K$ of characteristic $0$. If $\frak p$ is a prime ideal in $\cat O$, we will write $\cdvr{\cat O}{\frak p}$ for the complete DVR obtained from $\cat O$ by completing at $\frak p$. Unless otherwise indicated, tensor products will be taken over $\cat O$.
\end{Conv}

The completion maps $\cat O \to \cdvr{\cat O}{\frak p}$ along with the inclusion $\cat O \to K$ fit into a pullback square of $\cat O[G]$-modules with trivial $G$-action
\begin{equation}\label{eq:fracturesquare}
    \vcenter{
    \xymatrix{\cat O \ar[r] \ar[d] & \prod_{\frak p}\cdvr{\cat O}{\frak p} \ar[d] \\
    K \ar[r] & K\otimes\prod_{\frak p}\cdvr{\cat O}{\frak p}}
    }
\end{equation}
in which the product ranges over the non-zero prime ideals of $\cat O$. We will refer to \eqref{eq:fracturesquare} as the arithmetic fracture square of $\cat O$. For $\cat O = \bbZ$, it recovers the usual Hasse square. 

\begin{Thm}\label{thm:arithmeticrepltg}
The arithmetic fracture square \eqref{eq:fracturesquare} induces a pullback square of symmetric monoidal rigidly-compactly generated stable $\infty$-categories
\begin{equation}\label{eq:repfracture}
    \vcenter{
        \xymatrix{\Rep(G,\cat O) \ar[r] \ar[d] & \Rep(G,\prod_{\frak p}\cdvr{\cat O}{\frak p}) \ar[d] \\
        \Rep(G,K) \ar[r] & \Rep(G,K \otimes \prod_{\frak p}\cdvr{\cat O}{\frak p}).}
    }
\end{equation}
\end{Thm}
\begin{proof}
The adelic reconstruction theorem of Balchin and Greenlees \cite{balchingreenlees_adelic1}, applied to the derived category $\cat D(\cat O)$ of $\cat O$-modules, shows that \eqref{eq:fracturesquare} categorifies to a pullback square of symmetric monoidal $\infty$-categories\footnote{To be precise, Balchin and Greenlees work in the context of monoidal model categories. Passing to the underlying symmetric monoidal $\infty$-categories via \cite[\S4.1.7]{HALurie} results in the pullback square of \eqref{eq:derfracturesquare}.}
\begin{equation}\label{eq:derfracturesquare}
    \vcenter{
        \xymatrix{\cat D(\cat O) \ar[r] \ar[d] & \cat D(\prod_{\frak p}\cdvr{\cat O}{\frak p}) \ar[d] \\
        \cat D(K) \ar[r] & \cat D(K\otimes\prod_{\frak p}\cdvr{\cat O}{\frak p}).}
    }
\end{equation}
Let $\Cat_{\infty}$ be the $\infty$-category of $\infty$-categories. The Cartesian product equips $\Cat_{\infty}$ with the structure of a symmetric monoidal $\infty$-category (\cite[\S2.4.1]{HALurie}) and we denote the corresponding $\infty$-category of symmetric monoidal $\infty$-categories by $\CAlg(\Cat_{\infty})$. Viewing $BG$ as an $\infty$-groupoid and using the point-wise monoidal structure, the functor $\Fun(BG,-)\colon \Cat_{\infty} \to \Cat_{\infty}$ extends to a limit preserving functor
\[
\xymatrix{\Fun(BG,-)\colon \CAlg(\Cat_{\infty}) \ar[r] & \CAlg(\Cat_{\infty}).}
\]
The image of \eqref{eq:derfracturesquare} under this functor gives a pullback square of symmetric monoidal $\infty$-categories
\begin{equation}\label{eq:lsfracturesquare}
    \vcenter{
        \xymatrix{\Fun(BG,\cat D(\cat O)) \ar[r] \ar[d] & \Fun(BG,\cat D(\prod_{\frak p}\cdvr{\cat O}{\frak p})) \ar[d] \\
        \Fun(BG,\cat D(K)) \ar[r] & \Fun(BG,\cat D(K\otimes\prod_{\frak p}\cdvr{\cat O}{\frak p})).}
    }
\end{equation}
Finally, we obtain the pullback square \eqref{eq:repfracture} from \eqref{eq:lsfracturesquare} by applying the limit preserving (\cite[Prop.~4.6.1.11]{HALurie}) functor $\Ind((-)^{\dual})$ from $\CAlg(\Cat_{\infty})$ to the $\infty$-category of symmetric monoidal rigidly-compactly generated $\infty$-categories. 
\end{proof}

For $\cat O = \bbZ$, the next two results were proven in \cite[Prop.~5.2]{krause_picard} using \cref{prop:stmodtate}. His arguments extend to cover the more general case below as well; however, we give a different proof, deducing it from \cref{thm:arithmeticrepltg}.

\begin{Cor}\label{cor:arithmeticstmodltg1}
The completion maps $\cat O \to \cdvr{\cat O}{\frak p}$ induce a symmetric monoidal equivalence
\begin{equation}\label{eq:stmodfracture1}
    \vcenter{
        \xymatrix{\StMod(G,\cat O) \ar[r]^-{\sim} & \prod_{\frak p\mid |G|}\StMod(G, \cdvr{\cat O}{\frak p}).}
    }
\end{equation}
Here, the product is indexed on all prime ideals of $\cat O$ that divide the order of $G$.
\end{Cor}
\begin{proof}
Applying the finite localization functor $\rho^*$ from \eqref{eq:stmodquot} to the decategorification of the pullback square \eqref{eq:repfracture} yields a pullback square in $\StMod(G,\cat O)$:
\begin{equation}\label{eq:stmodfracturesquare}
    \vcenter{
        \xymatrix{\unit \simeq \rho^*(\cat O) \ar[r] \ar[d] & \rho^*(\prod_{\frak p}\cdvr{\cat O}{\frak p}) \ar[d] \\
        \rho^*(K) \ar[r] & \rho^*(K\otimes\prod_{\frak p}\cdvr{\cat O}{\frak p}).}
    }
\end{equation}
As in the proof of \cref{cor:stmodcharacteristic}, both domain and codomain of the bottom horizontal functor are trivial, hence it is an equivalence. It follows that the top horizontal functor in \eqref{eq:stmodfracturesquare} is an equivalence $\unit \simeq \rho^*(\prod_{\frak p}\cdvr{\cat O}{\frak p})$ of commutative ring objects. Separating the part where the order of $G$ is invertible thus gives an equivalence 
\[
\unit \simeq \rho^*\prod_{\frak p\mid |G|}\cdvr{\cat O}{\frak p} \times \rho^*\prod_{\frak p\nmid |G|}\cdvr{\cat O}{\frak p}.
\]
The order of $G$ is invertible in $\prod_{\frak p\nmid |G|}\cdvr{\cat O}{\frak p}$, so by another application of \cref{cor:stmodcharacteristic}, the second factor vanishes. Since the set of prime ideals in $\cat O$ dividing $|G|$ is finite, we get symmetric monoidal equivalences
\begin{align*}
\StMod(G,\cat O) & \simeq \Mod_{\StMod(G,\cat O)}(\unit) \\
& \simeq \Mod_{\StMod(G,\cat O)}(\rho^*\textstyle{\prod_{\frak p\nmid |G|}\cdvr{\cat O}{\frak p}}) \\
& \simeq \textstyle{\prod_{\frak p\nmid |G|}\Mod_{\StMod(G,\cat O)}(\cdvr{\cat O}{\frak p})} \\
& \simeq \textstyle{\prod_{\frak p\nmid |G|}\StMod(G,\cdvr{\cat O}{\frak p})},
\end{align*}
where the last equivalence follows from \cref{rem:etalecoeffbasechange}.
\end{proof}

\begin{Rem}
Alternatively, the pullback square \eqref{eq:repfracture} localizes to a pullback square of symmetric monoidal rigidly-compactly generated stable $\infty$-categories
\begin{equation}\label{eq:stmodfracturesquare}
    \vcenter{
        \xymatrix{\StMod(BG,\cat O) \ar[r] \ar[d] & \StMod(BG,\prod_{\frak p}\cdvr{\cat O}{\frak p}) \ar[d] \\
        \StMod(BG,K) \ar[r] & \StMod(BG,K\otimes\prod_{\frak p}\cdvr{\cat O}{\frak p}).}
    }
\end{equation}
Indeed, Verdier sequences are compatible with finite limits, which in turn follows from \cite[Thm.~I.3.3]{NikolausScholze18}. By \cref{cor:stmodcharacteristic}, the square \eqref{eq:stmodfracturesquare} collapses to the equivalence \eqref{eq:stmodfracture1}.
\end{Rem}

On compact objects, we obtain a refinement of the previous result: 

\begin{Cor}\label{cor:arithmeticstmodltg2}
The equivalence \eqref{eq:stmodfracture1} of \cref{cor:arithmeticstmodltg1} restricts to a symmetric monoidal equivalence
\begin{equation}\label{eq:stmodfracture2}
    \xymatrix{\stmod(G,\cat O)^{\natural} \ar[r]^-{\sim} & \prod_{\frak p\mid |G|}\stmod(G, \cdvr{\cat O}{\frak p}).}
\end{equation}
\end{Cor}
\begin{proof}
In light of \cref{cor:arithmeticstmodltg1} and \cref{rem:stmodnatural}, it remains to see that the tt-category $\stmod(G, \cdvr{\cat O}{\frak p})$ is idempotent complete. This can be shown as in the proof of Proposition 5.2 in \cite{krause_picard}, using that $\cdvr{\cat O}{\frak p}$ is $\pi$-adically complete.
\end{proof}

\begin{Rem}
The equivalence of \eqref{eq:stmodfracture2} is studied in more detail in work in progress of Grodal and Krause. In particular, as noted previously, $\stmod(G,\cat O)$ is in general not idempotent complete. They establish character-theoretic criteria for an object in $\stmod(G,\cat O)^{\natural}$ to already lie in $\stmod(G,\cat O)$.
\end{Rem}

\subsection{Stratification for the derived category of representations}\label{ssec:stratrepiffstmod}

Let $R$ be a Noetherian ring and $\cat D(R)$ its derived category of modules. In \cite{Neeman92a}, Neeman proved that $\cat D(R)$ is stratified over $\Spc(\Perf(R)) \cong \Spec(R)$. The goal of this subsection is to use descent to leverage Neeman's theorem to an equivalence between stratification for $\Rep(G,R)$ and stratification for $\StMod(G,R)$.

\begin{Lem}\label{lem:deraugconservative}
Let $R$ be a commutative ring, let $G$ be a finite group, and consider restriction and corestriction $i^*, i^!\colon \cat D(R[G]) \to \cat D(R)$. There is a natural isomorphism $i^* \simeq i^!$ and both functors are conservative.
\end{Lem}
\begin{proof}
Let $R[G] \to R$ be the canonical augmentation, induced by the inclusion $e \to G$ of the trivial group. Since $G$ is finite, induction from $R$ to $R[G]$ is naturally isomorphic to coinduction, so passing to right adjoints yields a natural isomorphism $i^* \simeq i^!$. An object in $D(R[G])$ is equivalent to $0$ if and only if all its homology groups are zero, and this is detected after forgetting the $G$-action. The claim follows.
\end{proof}

\begin{Thm}\label{thm:stratrepiffstmod}
Let $G$ be a finite group and $R$ a Noetherian commutative ring, then the following statements are equivalent:
    \begin{enumerate}
        \item[(1)] $\StMod(G,R)$ is stratified over $\Spec^h(H^*(G;R))\setminus\Spec(R)$.
        \item[(2)] $\Rep(G,R)$ is stratified over $\Spec^h(H^*(G;R))$.
    \end{enumerate}
\end{Thm}
\begin{proof}
The identification of the corresponding Balmer spectra is provided by \cref{thm:lau} and \cref{cor:spcstmod}, respectively. It is therefore enough to verify that $\StMod(G,R)$ is stratified if and only if $\Rep(G,R)$ is stratified.

By construction, the functor $\rho^*\colon \Rep(G,R) \to \StMod(G,R)$ appearing in \cref{prop:mainrecollement} is a finite localization. The implication $(2) \implies (1)$ thus follows from Zariski descent (\cref{thm:zariskidescent}). 

Conversely, suppose that $\StMod(G,R)$ is stratified. Write $i^*\colon \Rep(G,R) \to \cat D(R)$ for the symmetric monoidal forgetful functor from \cref{ssec:changeofgroups}. The corresponding base change functors $(i^*,i_*,i^!)$ together with the functors $(\rho^*,\rho_*,\rho^!)$ assemble into a triple of adjoints
\[
\xymatrix{\varphi^*=(\rho^*,i^*), \varphi^!=(\rho^!,i^!)\colon\Rep(G,R) \ar@<1ex>[r] \ar@<-1ex>[r] & \StMod(G,R) \times \cat D(R) \ar[l] \noloc \varphi_* = \rho_*\times i_*.}
\]
In other words, $\varphi^*$ has a right adjoint $\varphi_*$ which admits a further right adjoint $\varphi^! = (\rho^!,i^!)$, constructed as above. The idea is to apply nil-descent  to $\varphi^*$. 

To this end, we have to verify Conditions (a) and (b) of \cref{thm:nildescent}. We first claim that $\varphi^*$ is conservative. Indeed, suppose $\varphi^*(X) \simeq 0$ for some $X \in \Rep(G,R)$. By considering the first factor, we see that $X$ is in the kernel of $\rho^*$, which identifies with $\cat D(R[G])$ by \cref{prop:mainrecollement}. In light of \cref{lem:deraugconservative}, this forces $X \simeq 0$, as claimed. Similarly, consider $X$ with $\varphi^!(X)=0$. In particular, $X$ is in the kernel of $\rho^!$. By abstract local duality in the sense of \cite[Thm.~2.21]{bhv1}, this implies that $X \simeq 0$ if and only if $X \otimes R[G] \simeq 0$. But $X \otimes R[G]$ is contained in $\Loco{R[G]} \simeq \cat D(R[G])$, so we compute
\[
i^*(X \otimes R[G]) \simeq i^*(X) \otimes i^*(R[G]) \simeq i^!(X) \otimes i^*(R[G]) \simeq 0
\]
as a consequence of \cref{lem:deraugconservative}. Therefore, $X \otimes R[G] \simeq 0$ as desired.

Next, we have to verify that $\varphi_*$ is conservative, which is equivalent to saying that both $\rho_*$ and $i_*$ are conservative. In fact, $\rho_*$ is fully faithful, while the claim for $i_*$ follows from \cref{prop:restrictionetale}. We have thus verified Condition (a).

As for Condition (b), the restriction of $\varphi^*$ to compact objects induces a map on Balmer spectra
\[
\xymatrix{\Spc(\stmod(G,R)) \coprod \Spc(\Perf(R)) \ar[r] & \Spc(\rep(G,R))}
\]
which is a spectral bijection in light of \cref{cor:spcstmod}. We are therefore in the position to apply \cref{thm:nildescent}: By assumption, $\StMod(G,R)$ is stratified. Since $\cat D(R)$ is stratified as well by Neeman's theorem \cite[Thm.~2.8]{Neeman92a}, stratification descends to $\Rep(G,R)$, as desired.
\end{proof}

\subsection{Stratification for the stable module category}\label{ssec:stratstmod}

We are now ready to put the pieces together to establish stratification for the integral stable module category of a finite group. Throughout this section, \cref{conv:o} will be in place. As a proof of concept, we first give an independent proof of the main theorem of \cite{BensonIyengarKrause11a}:

\begin{Thm}[Benson--Iyengar--Krause]
Let $G$ be a finite group and $k$ a field of characteristic $p>0$. The category $\StMod(G,k)$ is stratified over $\Proj(H^*(G;k))$.
\end{Thm}
\begin{proof}
Via \'etale descent in the form of \cref{thm:descenttoelab}, stratification for $\StMod(G,k)$ follows from \cref{thm:elabbikstrat}. The identification of the spectrum is given in \cref{ex:bcr} or may alternatively be deduced from \cref{thm:genquillenstrat} and \cref{cor:quillenstratification}.
\end{proof}

\begin{Thm}\label{thm:stratstmod}
Let $G$ be a finite group and let $\cat O$ be as in \cref{conv:o}. The category $\StMod(G,\cat O)$ is stratified over $\left(\Spec^h(H^*(G;\cat O))\setminus\Spec(\cat O)\right)$.
\end{Thm}
\begin{proof}
Consider an elementary abelian subgroup $E$ of $G$. If $\frak p$ is a non-zero prime ideal in $\cat O$, then $\cdvr{\cat O}{\frak p}$ is a complete DVR of mixed characteristic $(0,p)$ for some prime~$p$. By \cref{thm:elabstmodstratification}, the category $\StMod(E,\cdvr{\cat O}{\frak p})$ is stratified. Note that the set of prime ideals in $\cat O$ which divide the order of $G$ is finite. It therefore follows from \cref{cor:arithmeticstmodltg1} and \cref{cor:stratfiniteproduct} that $\StMod(E,\cat O)$ is stratified for any $E \in \cat E(G)$.

We can thus apply \cref{thm:descenttoelab} to conclude that $\StMod(G,\cat O)$ is stratified as well. The Balmer spectrum of $\stmod(G,\cat O)$ was determined in \cref{cor:spcstmod}.
\end{proof}

\begin{Cor}\label{cor:stratstmod}
The universal support function composed with Balmer's comparison map induces a bijection
\[
\Supp\colon\begin{Bmatrix}
\text{Localizing tensor ideals} \\
\text{of } \cat \StMod(G,\cat O)
\end{Bmatrix} 
\xymatrix@C=2pc{ \ar[r]^-{\sim} &}
\begin{Bmatrix}
\text{Subsets of}  \\
\Spec^h(H^*(G;\cat O))\setminus\Spec(\cat O)
\end{Bmatrix}.
\]
Moreover, the telecope conjecture holds in $\StMod(G,\cat O)$, so that we obtain bijections
\[
\begin{Bmatrix}
\text{Smashing ideals} \\
\text{of } \cat \StMod(G,\cat O)
\end{Bmatrix} 
\simeq
\begin{Bmatrix}
\text{Thick tensor ideals} \\
\text{of } \cat \stmod(G,\cat O)
\end{Bmatrix} 
\simeq
\begin{Bmatrix}
\text{Specialization closed subsets}  \\
\text{of }\Spec^h(H^*(G;\cat O))\setminus\Spec(\cat O)
\end{Bmatrix}.
\]
Finally, cohomological support satisfies the tensor product formula:
\[
\Supp(M \otimes N) = \Supp(M) \cap \Supp(N)
\]
for any $M,N \in \StMod(G,\cat O)$.
\end{Cor}
\begin{proof}
The first statement is a consequence of \cref{thm:stratstmod} and \cref{cor:spcstmod}, while the second one follows from  \cref{prop:gentelescope}. The tensor product formula for support then holds by \cref{prop:tensorproductformula}. 
\end{proof}

In light of \cref{thm:stratrepiffstmod}, stratification for the stable module category is equivalent to stratification for the derived category of representations. 

\begin{Thm}\label{thm:stratrep}
Let $G$ be a finite group and let $\cat O$ be as in \cref{conv:o}. The category $\Rep(G,\cat O)$ is stratified over $\Spec^h(H^*(G;\cat O))$: The universal support function composed with Balmer's comparison map induces a bijection
\[
\Supp\colon\begin{Bmatrix}
\text{Localizing tensor ideals} \\
\text{of } \cat \Rep(G,\cat O)
\end{Bmatrix} 
\xymatrix@C=2pc{ \ar[r]^-{\sim} &}
\begin{Bmatrix}
\text{Subsets of}  \\
\Spec^h(H^*(G;\cat O))
\end{Bmatrix}.
\]
Moreover, the telecope conjecture holds in $\Rep(G,\cat O)$, so that we obtain bijections
\[
\begin{Bmatrix}
\text{Smashing ideals} \\
\text{of } \cat \Rep(G,\cat O)
\end{Bmatrix} 
\simeq
\begin{Bmatrix}
\text{Thick tensor ideals} \\
\text{of } \cat \rep(G,\cat O)
\end{Bmatrix} 
\simeq
\begin{Bmatrix}
\text{Specialization closed}  \\
\text{subsets of } \Spec^h(H^*(G;\cat O))
\end{Bmatrix},
\]
and support satisfies the tensor product formula for any $M,N \in \Rep(G,\cat O)$:
\[
\Supp(M \otimes N) = \Supp(M) \cap \Supp(N).
\]
\end{Thm}
\begin{proof}
By \cref{thm:stratrepiffstmod} and \cref{thm:stratstmod}, the tt-category $\Rep(G,\cat O)$ is stratified. Lau's theorem (\cref{thm:lau}) identifies the Balmer spectrum and the telescope conjecture follows from  \cref{prop:gentelescope}. Finally, the tensor product formula is a consequence of stratification and \cref{prop:tensorproductformula}.
\end{proof}

\subsection{Stratification for the category of representations}\label{ssec:stratreps}

In this final subsection, we deduce a generic classification theorem for $\cat O$-linear $G$-representations. Recall that $\projmod(G,\cat O)$ denotes the Frobenius category of $\cat O[G]$-modules whose underlying $\cat O$-modules are projective. The notion of Serre and localizing subcategories generalizes from the setting of abelian categories to exact categories.

\begin{Def}\label{def:exlocalizing}
Let $\cat E$ be an exact category. A \emph{thick} subcategory $\cat S \subseteq \cat E$ is a full subcategory which satisfies the following closure properties:
    \begin{itemize}
        \item $\cat S$ is closed under direct summands. 
        \item If $0\to X \to Y \to Z \to 0$ is an exact sequence in $\cat E$ such that two terms of it are in $\cat S$, then so is the third.  
    \end{itemize}
For an exact category $\cat E$ which admits set-indexed coproducts, a thick subcategory of $\cat E$ is called a \emph{localizing} subcategory if it is also closed under set-indexed coproducts. If $\cat E$ is additionally equipped with an exact symmetric monoidal structure $\otimes$, we say that $\cat S$ is a thick tensor ideal or localizing tensor ideal if it is a thick subcategory or localizing subcategory which is additionally closed under $\otimes$, respectively. 
\end{Def}

\begin{Lem}\label{lem:alglocalizingideals}
The functor $\projmod(G,R) \to \StMod(G,R)$ of \eqref{eq:algmodelfunctor} gives a bijection:
\[
\begin{Bmatrix}
\text{Non-zero localizing tensor} \\
\text{ideals of } \projmod(G,R)
\end{Bmatrix} 
\xymatrix@C=2pc{ \ar[r]^-{\sim} &}
\begin{Bmatrix}
\text{Localizing tensor ideals}  \\
\text{of } \StMod(G,R)
\end{Bmatrix}.
\]
\end{Lem}
\begin{proof}
This is an adaptation of the proof of \cite[Prop.~2.1]{BensonIyengarKrause11a}, see also \cite[Thm.~1]{krausestevenson_stablederived} for the thick tensor ideal version of this result. Indeed, suppose $\cat S$ is a non-zero localizing tensor ideal of $\projmod(G,R)$. For any non-trivial $X$ in $\cat S$, the tensor product $R[G] \otimes X \in \cat S$ is projective and non-zero, see \cref{lem:exactstructure}. It follows that $R[G] \in \cat S$, so all projective objects are contained in $\cat S$. The desired bijection is then a consequence of \cref{prop:algebraicmodel}.
\end{proof}

\begin{Thm}\label{thm:stratex}
Cohomological support induces a bijection
\[
\begin{Bmatrix}
\text{Non-zero localizing tensor} \\
\text{ideals of } \cat \projmod(G,\cat O)
\end{Bmatrix} 
\simeq
\begin{Bmatrix}
\text{Subsets of} \\
\Spec^h(H^*(G;\cat O))\setminus\Spec(\cat O)
\end{Bmatrix}.
\]
\end{Thm}
\begin{proof}
This is a consequence of \cref{cor:stratstmod} and \cref{lem:alglocalizingideals}.
\end{proof}

\begin{Rem}\label{rem:tstex}
Let $G$ be a finite group and let $R$ be a regular Noetherian commutative ring. It follows from \cref{cor:spcstmod} and \cite[Thm.~1]{krausestevenson_stablederived} that
\[
\begin{Bmatrix}
\text{Non-zero thick} \\
\text{ideals of } \cat \projmod(G,R)^{\fp}
\end{Bmatrix} 
\simeq
\begin{Bmatrix}
\text{Specialization closed subsets of} \\
\Spec^h(H^*(G;R))\setminus\Spec(R)
\end{Bmatrix}.
\]
In light of this, it would be interesting to generalize \cref{thm:stratex} from coefficients in $\cat O$ to more general rings. 
\end{Rem}

\begin{Rem}\label{rem:costratification}
Utilizing the theory developed in ongoing work with Castellana, Heard, and Sanders \cite{bchs1}, it is plausible that the ideas of the present paper can also be used to establish costratification for $\StMod(G,\cat O)$ and $\Rep(G, \cat O)$.
\end{Rem}

\bibliographystyle{alpha}\bibliography{TG-articles}

\end{document}